\documentclass[a4paper,10pt]{amsart}

\usepackage[T1]{fontenc} \usepackage[utf8]{inputenc}
\usepackage{lmodern}

\usepackage{enumerate}

\usepackage{amsfonts,amsmath,amstext,amsbsy,amssymb}
\usepackage{amsbsy,amsopn,amsthm,amscd,amsxtra}

\usepackage{latexsym,mathrsfs}

\usepackage{mathabx} \newcommand{\carr}{\righttoleftarrow}

\usepackage{hyperref}
\RequirePackage[dvipsnames]{xcolor} 
\definecolor{halfgray}{gray}{0.55} 
\definecolor{webgreen}{rgb}{0,0.5,0}
\definecolor{webbrown}{rgb}{.6,0,0} \hypersetup{%
  colorlinks=true, linktocpage=true, pdfstartpage=3,
  pdfstartview=FitV,%
  breaklinks=true, pdfpagemode=UseNone, pageanchor=true,
  pdfpagemode=UseOutlines,%
  plainpages=false, bookmarksnumbered, bookmarksopen=true,
  bookmarksopenlevel=1,%
  hypertexnames=true,
  pdfhighlight=/O,
  urlcolor=webbrown, linkcolor=RoyalBlue,
  citecolor=webgreen, 
  pdftitle={Rotational deviations},%
  pdfauthor={Alejandro Kocsard},%
  pdfsubject={2000 MAthematical Subject Classification: Primary:},%
  pdfkeywords={},%
  pdfcreator={pdfLaTeX},%
  pdfproducer={LaTeX with hyperref}%
}

\newcommand{\abs}[1]{\left\lvert{#1}\right\rvert}
\newcommand{\norm}[1]{\left\|{#1}\right\|}

\newcommand{\scprod}[2]{\left\langle{#1},{#2}\right\rangle}

\newcommand{\bb}{\mathbb} 
\newcommand{\mc}{\mathcal} 

\newcommand{\R}{\mathbb{R}}\newcommand{\N}{\mathbb{N}}
\newcommand{\Z}{\mathbb{Z}}\newcommand{\Q}{\mathbb{Q}}
\newcommand{\T}{\mathbb{T}}
\newcommand{\A}{\mathbb{A}}
\newcommand{\Ss}{\mathbb{S}}

 \newcommand{\ie}{i.e.\ }

\newtheorem{theorem}{Theorem}[section]
\newtheorem{proposition}[theorem]{Proposition}

\newtheorem{corollary}[theorem]{Corollary}

\theoremstyle{definition}

\theoremstyle{remark} \newtheorem{remark}[theorem]{Remark}

\newtheorem*{theorem*}{Theorem}

\newcommand{\Homeo}[1]{\mathrm{Homeo}_{#1}}

 \newcommand{\dd}{\:\mathrm{d}}

\DeclareMathOperator{\M}{\mathfrak{M}} 
\DeclareMathOperator{\Per}{Per} \DeclareMathOperator{\Fix}{Fix}

\newcommand{\Thomeo}{\widetilde{\mathrm{Homeo}_0}}

\newcommand{\wh}{\widehat}
\newcommand{\I}{\mathscr{I}}

\DeclareMathOperator{\Ad}{Ad} 
\DeclareMathOperator{\inter}{int}

\newcommand{\F}{\mathscr{F}} \newcommand{\cc}{\mathrm{cc}}
\newcommand{\pr}[1]{\mathrm{pr}_{#1}} 

\title[Rotational deviations]{Rotational deviations and invariant
  pseudo-foliations for periodic point free torus homeomorphisms}

\author[A. Kocsard]{Alejandro Kocsard}

\email{akocsard@id.uff.br}

\address{IME - Universidade Federal Fluminense. Rua Prof. Marcos
  Waldemar de Freitas Reis, S/N. Bloco H, $4^\circ$
  andar. 24.210-201, Gragoatá, Niterói, RJ, Brasil}

\author[F. Rodrigues]{Fernanda Pereira Rodrigues}

\email{fernandapr@id.uff.br}

\address{IME - Universidade Federal Fluminense. Rua Prof. Marcos
  Waldemar de Freitas Reis, S/N. Bloco H, $4^\circ$
  andar. 24.210-201, Gragoatá, Niterói, RJ, Brasil}

\date{\today}

\begin{document}

\begin{abstract}
  This article deals with directional rotational deviations for
  non-wandering periodic point free homeomorphisms of the $2$-torus
  which are homotopic to the identity. We prove that under mild
  assumptions, such a homeomorphism exhibits uniformly bounded
  rotational deviations in some direction if and only if it leaves
  invariant a \emph{pseudo-foliation,} a notion which is a slight
  generalization of classical one-dimensional foliations.

  To get these results, we introduce a novel object called
  \emph{$\tilde\rho$-centralized skew-product} and their associated
  \emph{stable sets at infinity.}
\end{abstract}

\maketitle

\section{Introduction}
\label{sec:intro}

We denote by $\Homeo0(\T^d)$ the space of homeomorphisms of the
$d$-dimensional torus $\T^d$ which are homotopic to the identity. The
main dynamical invariant for systems given by such a map is the so
called \emph{rotation set.} Given a lift $\tilde f\colon\R^d\carr$ of
a homeomorphism $f\in\Homeo0(\T^d)$, one defines its \emph{rotation
  set} by
\begin{displaymath}
  \rho(\tilde f):=\left\{\rho\in\R^d : \exists n_i\uparrow+\infty,\
    z_i\in\R^d, \ \frac{\tilde f^{n_i}(z_i)-z_i}{n_i}\to\rho,\
    \text{as } i\to\infty\right\}.
\end{displaymath}

This invariant was originally defined by Poincaré for the
one-dimensional case (\ie $d=1$) in his celebrated work
\cite{Poincare1880memoire}. In such a case, the rotation set reduces
to a point, the so called \emph{rotation number}, and a simple but
fundamental property holds (see for instance \cite[page
21]{HermanSurLaConj}):
\begin{equation}
  \label{eq:bounded-dev-d1}
  \abs{\tilde f^n(z)-z -n\rho(\tilde f)}\leq 1,\quad\forall n\in\Z,
  \forall z\in\R.
\end{equation}
That is, every $\tilde f$-orbit exhibits \emph{uniformly bounded
  rotational deviations} with respect to the rigid rotation (or
translation) $z\mapsto z+\rho(\tilde f)$.

In higher dimensions the situation dramatically changes: in fact, for
$d\geq 2$ different orbits can exhibit different rotation vectors and,
moreover, there can exist points with non-well-defined rotation
vector. However, the rotation set $\rho(\tilde f)$ is always
non-empty, compact and connected; and for $d=2$, Misiurewicz and
Ziemian~\cite{MisiurewiczZiemian} showed it is also convex. So, in the
two-dimensional case the elements of $\Homeo0(\T^2)$ can be classified
according to the following trichotomy:
\begin{enumerate}[(i)]
\item \label{enum:pseudo-rot} $\rho(\tilde f)$ is just a point, and in
  this case we say $f$ is a \emph{pseudo-rotation;}
\item \label{enum:segment} $\rho(\tilde f)$ is a non-degenerate line
  segment;
\item \label{enum:fat} $\rho(\tilde f)$ has non-empty interior.
\end{enumerate}

In this paper we shall concentrate on the study of \emph{rotational
  deviations} in dimension two. In this case, as a generalization of
\eqref{eq:bounded-dev-d1}, one says that $f$ exhibits \emph{uniformly
  bounded rotational deviations} when
\begin{displaymath}
  \sup_{z\in\R^2} \sup_{n\in\Z} d\big(\tilde f^n(\tilde z) - \tilde z,
  n\rho(\tilde f)\big) <\infty.
\end{displaymath}

When the rotation set $\rho(\tilde f)$ has non-empty interior Le
Calvez and Tal~\cite{LeCalvezTalForcingTheory} has recently shown that
$f$ exhibits uniformly bounded rotational deviations. This result had
been previously gotten by
Addas-Zanata~\cite{Addas-ZanataUnifBoundDiff} for smooth
diffeomorphisms, and Dávalos~\cite{DavalosAnnularMapsTorus} in some
particular cases.

When the rotation set $\rho(\tilde f)$ has empty interior, \ie it
satisfies either condition \eqref{enum:pseudo-rot} or
\eqref{enum:segment} of above trichotomy, one can consider
\emph{directional rotational deviations:} if $v\in\Ss^1$ denotes a
unit vector such that $\rho(\tilde f)$ is contained in a straight line
perpendicular to $v$, $f$ is said to exhibit \emph{uniformly bounded
  $v$-deviations} when there exists a constant $M>0$ such that
\begin{displaymath}
  \scprod{\tilde f^n(z)-z -  n\rho}{v}\leq M, \quad\forall z\in\R^2,\
  \forall n\in\Z,
\end{displaymath}
for some (and hence, any) $\rho\in\rho(\tilde f)$.

When $f$ is a pseudo-rotation, \ie $\rho(\tilde f)$ is just a point,
one can study rotational deviations in any direction. Koropecki and
Tal has shown in \cite{KoroTalBoundUnbound} that ``generic''
area-preserving rational pseudo-rotations, \ie satisfying
$\rho(\tilde f)\subset\Q^2$, exhibit uniformly bounded deviations in
every direction (see also \cite{LeCalvezTalForcingTheory}), but there
are some ``exotic'' rational pseudo-rotations with unbounded
deviations in every direction
\cite{KoropeckiTalAreaPrsIrrotDiff}. Contrasting with these results,
in \cite{KocKorFoliations} the first author and Koropecki proved that
generic diffeomorphisms in the closure of the conjugacy class of
rotations on $\T^2$ are irrational pseudo-rotations exhibiting
unbounded deviations in every direction.

Regarding the remaining case \eqref{enum:segment} of above
trichotomy, Dávalos~\cite{DavalosAnnularMapsTorus} have proved that
any $f\in\Homeo0(\T^2)$ whose rotation set is a non-trivial vertical
segment and contains rational points, exhibits uniformly bounded
horizontal deviations. This result had been proven by
Guelman, Koropecki and Tal~\cite{GuelKorTalAnnAreaPresToral} in the
area-preserving setting.

As the reader could have already noticed, in all above boundedness
results the homeomorphisms have periodic points, and in fact, these
orbits play a fundamental role in their proofs.

The fundamental purpose of this paper consists in pursuing the study
of directional rotational deviations for periodic point free
homeomorphisms and its topological and geometric consequences.

It is well-known that any $2$-torus homeomorphism in the identity
isotopy class that leaves invariant a (non-singular one-dimensional)
foliation exhibits uniformly bounded $v$-deviations, for some
$v\in\Ss^1$ whose direction is completely determined by the asymptotic
behavior of the foliation; and in particular, its rotation set has
empty interior.

However, as it is shown in \cite{BeguinCrovisierJaegerADynDecomp},
there are smooth minimal area-preserving pseudo-rotations exhibiting
uniformly bounded horizontal deviations, but not preserving any
foliation.

One of the main results of this paper, that emerges from the
combination of Theorems~\ref{thm:inv-pseudo-fol-implies-bounded-dev}
and \ref{thm:pseudo-fol-vs-bounded-rot-dev}, establishes that, under
some mild and natural conditions, an element of $\Homeo0(\T^2)$
exhibits uniformly bounded $v$-deviations for some $v\in\Ss^1$ if and
only if it leaves invariant a \emph{pseudo-foliation.}  Torus
pseudo-foliations, which are a slight generalization of classical
one-dimensional foliations, are defined in
\S\ref{sec:tor-foliations}. Pseudo-foliations have been used in
\cite{KocMinimalHomeosNotPseudo} to show that any minimal
homeomorphism which is not a pseudo-rotation is topologically mixing,
and in \cite{KoropeckiPasseggiSambarino} to prove a particular case of
the so called Franks-Misiurewicz conjecture (even though the authors
did not explicitly use this name).


The main character of this work is a new class of dynamical systems on
$\T^2\times\R^2$, called \emph{$\tilde \rho$-centralized
  skew-products}, that are associated to each homeomorphism of $\T^2$
which is isotopic to the identity. In \S~\ref{sec:fibered-stable-sets}
we define the \emph{stable sets at infinity} associated to these
skew-products and in \S~\ref{sec:dir-dev-ppf-homeos} we stablish some
important results about the topology of these sets regarding the
directional rotational devitions of the original system. Stable sets
at infinity are used to contruct the invariant pseudo-foliation of
Theorem~\ref{thm:pseudo-fol-vs-bounded-rot-dev}.

The paper is organized as follows: in \S~\ref{sec:pelim-notat} we fix
the notation that will be used all along the article and recall some
previous results.  In \S~\ref{sec:induced-skew-prod} we introduce the
so called \emph{$\tilde\rho$-centralized skew-products} and their
associated \emph{stable sets at infinity.} We prove some elementary
properties of these sets and use them to get some fundamental results
about rotational deviations, like
Theorem~\ref{thm:v-bound-dev-iff--v-bound-dev} and
Corollary~\ref{cor:unif-bound-dev-iff-Hm-in-Lambda}, showing that
uniformly $v$- and $(-v)$-deviations are equivalent. Then, in
\S~\ref{sec:dir-dev-ppf-homeos} we establish some relations between
the topology of \emph{stable sets at infinity} and \emph{à priori}
boundedness of rotational deviations for non-wandering periodic point
free homeomorphisms. Then, finally in \S~\ref{sec:tor-foliations} we
introduce the new concepts of \emph{torus pseudo-foliations} we show
that under some natural and mild hypotheses, a non-wandering periodic
point free homeomorphism exhibits uniformly bounded rotational
deviations in some direction if and only if it leaves invariant a
torus pseudo-foliation.


\subsection*{Acknowledgments}
\label{sec:acknowledgments}

A.K. was partially supported by FAPERJ (Brazil) and CNPq
(Brazil). F. P.-R. was supported by CAPES (Brazil).

We are very indebted to Andrés Koropecki, Fábio Tal, Rafael Potrie,
and Salvador Addas-Zanata for many insightful discussions about this
work.

\section{Preliminaries and notations}
\label{sec:pelim-notat}

\subsection{Maps, topological spaces and groups}
\label{sec:maps-top-groups}

Given a map $f\colon X\carr$, its set of fixed points will be denoted
by $\Fix(f)$. We shall write $\Per(f):=\bigcup_{n\geq 1}\Fix(f^n)$ for
the set of periodic points. The map $f$ is said to be \emph{periodic
  point free} whenever $\Per(f)=\emptyset$.  If $A\subset X$ denotes
an arbitrary subset, we define its positively maximal $f$-invariant
subset by
\begin{equation}
  \label{eq:pos-max-inv-set-def}
  \I_f^+(A):=\bigcap_{n\geq 0} f^{-n}(A).
\end{equation}
When $f$ is bijective, we can also define its maximal $f$-invariant
subset by
\begin{equation}
  \label{eq:max-inv-set-def}
  \I_f(A):=\I_f^+(A)\cap\I_{f^{-1}}^+(A) = \bigcap_{n\in\Z} f^n(A). 
\end{equation}

When $X$ is a topological space, a homeomorphism $f\colon X\carr$ is
said to be \emph{non-wandering} when for every non-empty open set
$U\subset X$, there exists $n\geq 1$ such that
$f^n(U)\cap U\neq\emptyset$.  On the other hand, $f$ is said to be
\emph{minimal} when every $f$-orbit is dense in $X$.

Given any $A\subset X$, we write $\inter(A)$ to denote its interior,
$\bar A$ for its closure and $\partial_X A$ for its boundary inside
$X$. When $A$ is connected, we write $\cc(X,A)$ for the connected
component of $X$ containing $A$. As usual, we write $\pi_0(X)$ to
denote the set of all connected components of $X$.

If $(X,d)$ is a metric space, the open ball of radius $r>0$ and center
at $x\in X$ will be denoted by $B_r(x)$. Given an arbitrary non-empty
set $A\subset X$, we define its \emph{diameter} by
\begin{displaymath}
  \mathrm{diam}(A):=\sup_{x,y\in A} d(x,y),
\end{displaymath}
and given any point $x\in X$, we write
\begin{displaymath}
  d(x,A):= \inf_{y\in A} d(x,y).
\end{displaymath}

We also consider the space of compact subsets
\begin{displaymath}
  \mc{K}(X):=\{K\subset X : K \text{ is non-empty and compact}\}.
\end{displaymath}
and we endow this space with its Hausdorff distance $d_H$ (induced by
$d$) defined by
\begin{displaymath}
  d_H(K_1,K_2):=\max\left\{\max_{p\in K_1} d(p,K_2), \max_{q\in K_2}
    d(q,K_1) \right\},\quad\forall K_1,K_2\in\mc{K}(X). 
\end{displaymath}
It is well-known that $(\mc{K}(X),d_H)$ is compact whenever $(X,d)$ is
compact itself.

Given a locally compact non-compact topological space $Y$, we write
$\wh{Y}:=Y\sqcup\{\infty\}$ for the one-point compactification of
$Y$. If $A\subset Y$ is an arbitrary subset, $\wh{A}$ will denote its
closure inside the space $\wh{Y}$, and given any continuous proper map
$f\colon Y\carr$, its unique extension to $\wh{Y}$ (that fixes the
point at infinity) will be denoted by $\wh{f}\colon\wh{Y}\carr$.

Whenever $M_1,M_2,\ldots M_n$ denote $n$ arbitrary sets, we shall use
the generic notation
$\pr{i}\colon M_1\times M_2\times\ldots\times M_n\to M_i$ to denote
the $i^\mathrm{th}$-coordinate projection map.

\subsection{Euclidean spaces and tori}
\label{sec:euclidean-spaces-tori}

We consider $\R^d$ endowed with its usual Euclidean structure denoted
by $\scprod{\cdot}{\cdot}$. We write $\norm{v}:=\scprod{v}{v}^{1/2}$,
for any $v\in\R^d$. The unit $(d-1)$-sphere is denoted by
$\Ss^{d-1}:=\{v\in\R^d : \norm{v}=1\}$.  For any
$v\in\R^d\setminus\{0\}$ and each $r\in\R$, we define the
\emph{half-space}
\begin{equation}
  \label{eq:Hvr-half-spaces-def}
  \bb{H}_{r}^v:=\left\{z\in\R^d: \scprod{z}{v}>r\right\}.
\end{equation}

For $d=2$, given any $v=(a,b)\in\R^2$, we define $v^\perp:=(-b,a)$. We
also introduce the following notation for straight lines: given any
$r\in\R$ and $v\in\Ss^1$,
\begin{equation}
  \label{eq:lrv-straight-line}
  \ell_r^v:= rv+\R v^\perp = \{ rv + tv^\perp : t\in\R \}.
\end{equation}

We say that $v\in\Ss^1$ has \emph{rational slope} when there exists
some $t\in\R\setminus\{0\}$ such that $tv\in\Z^2$; otherwise, $v$ is
said to have \emph{irrational slope.}

We will need the following notation for strips on $\R^2$: given
$v\in\Ss^1$ and $s>0$ we define the strip
\begin{equation}
  \label{eq:Asv-strip-def}
  \A^v_s:=\bb{H}^v_{-s}\cap\bb{H}^{-v}_{-s} =
  \{z\in\R^2 : -s<\scprod{z}{v}<s\}.
\end{equation}

Given any $\alpha\in\R^d$, $T_\alpha$ denotes the translation
$T_\alpha\colon z\mapsto z+\alpha$ on $\R^d$.

The $d$-dimensional torus $\R^d/\Z^d$ will be denoted by $\T^d$ and we
write $\pi\colon\R^d\to\T^d$ for the natural quotient
projection. Given any $\alpha\in\T^d$, we write $T_\alpha$ for the
torus translation $T_\alpha\colon\T^d\ni z\mapsto z+\alpha$.

As usual, a point $\alpha\in\R^d$ is said to be \emph{rational} when
$\alpha\in\Q^d$ and is said to be \emph{totally irrational} when
$T_{\pi(\alpha)}$ is a minimal homeomorphism of $\T^d$.

\subsection{The boundary at infinity of planar sets}
\label{sec:boundary-at-infinity}

Given a non-empty set $A\subset\R^2$ and a point $v\in\Ss^1$, we say
that $A$ \emph{accumulates in the direction $v$ at infinity} if there
is sequence $\{x_n\}_{n\geq 0}$ of points in $A$ such that
\begin{displaymath}
  \lim_{n\to\infty}\norm{x_n}=\infty,\quad\text{and} \quad
  \lim_{n\to\infty} \frac{x_n}{\norm{x_n}}=v.
\end{displaymath}
Then we can define the \emph{boundary of $A$ at infinity} as the set
$\partial_\infty A\subset\Ss^1$ consisting of all $v\in\Ss^1$ such
that $A$ accumulates in the direction $v$ at infinity.

\subsection{Surface topology}
\label{sec:surface-topology}

Let $S$ denote an arbitrary connected surface, \ie a two-dimensional
topological connected manifold.

An open subset of $S$ is said to be a \emph{topological disk} when it
is homeomorphic to the open unitary disk $\{z\in\R^2 : \norm{z}<1\}$.
Similarly, a \emph{topological annulus} is an open subset of $S$
homeomorphic to $\T\times\R$.

An \emph{arc} on $S$ is a continuous map $\alpha\colon [0,1]\to S$ and
a \emph{loop} on $S$ is a continuous map $\gamma\colon\T\to S$.

An open non-empty subset $U\subset S$ is said to be \emph{inessential}
when every loop in $U$ is contractible in $S$; otherwise it is said to
be \emph{essential.} An arbitrary subset $V\subset S$ is said to be
\emph{inessential} when there exists an inessential open set
$U\subset S$ such that $V\subset U$.  On the other hand, we say that
$V$ is \emph{fully essential} when $S\setminus V$ is inessential.

We say a subset $A\subset S$ is \emph{annular} when it is an open
topological annulus and none connected component of $S\setminus A$ is
inessential.

Finally, let us recall two classical results about fixed point free
orientation preserving homeomorphisms:

\begin{theorem}[Brouwer's translation theorem
  \cite{FathiOrbCloBrouwer}]
  \label{thm:Brouwer}
  Let $f\colon\R^2\carr$ be an orientation preserving homeomorphism
  such that $\Fix(f)=\emptyset$. Then, every $x\in\R^2$ is wandering
  for $f$, \ie there exists a neighborhood $U$ of $x$ such that
  $f^n(U)\cap U=\emptyset$, for every $n\in\Z\setminus\{0\}$.
\end{theorem}

\begin{theorem}[Corollary 1.3 in \cite{FranksNewProofBrouwer}]
  \label{thm:Franks-free-disk}
  Let $f\colon\R^2\carr$ be an orientation preserving homeomorphism
  such that $\Fix(f)=\emptyset$, and $D\subset\R^2$ be an open
  topological disk. Let us suppose $f(D)\cap D=\emptyset$. Then,
  $f^n(D)\cap D=\emptyset$, for every $n\in\Z\setminus\{0\}$.
\end{theorem}

\subsection{Groups of homeomorphisms}
\label{sec:groups-homeos}

Given any topological manifold $M$, $\Homeo{}(M)$ denotes the group of
homeomorphisms from $M$ onto itself. The subgroup formed by those
homeomorphisms which are homotopic to the identity map $id_M$ will be
denoted by $\Homeo0(M)$.

We define the subgroup $\Thomeo(\T^d)<\Homeo0(\R^d)$ by
\begin{displaymath}
  \Thomeo(\T^d):=\left\{\tilde f\in\Homeo0(\R^d) : \tilde f-id_{\R^d}\in
    C^0(\T^d,\R^d)\right\}.
\end{displaymath}
Notice that in this definition, we are identifying the elements of
$C^0(\T^d,\R^d)$ with those $\Z^d$-periodic functions from $\R^d$ into
itself.

Making some abuse of notation, we also write
$\pi\colon\Thomeo(\T^d)\to\Homeo0(\T^d)$ for the map that associates
to each $\tilde f$ the only torus homeomorphism $\pi\tilde f$ such
that $\tilde f$ is a lift of $\pi\tilde f$. Notice that with our
notations, it holds $\pi T_\alpha =T_{\pi(\alpha)}\in\Homeo0(\T^d)$,
for every $\alpha\in\R^d$.

Given any $\tilde f\in\Thomeo(\T^d)$, we define its \emph{displacement
  function} by
\begin{equation}
  \label{eq:Delta-cocycle}
  \Delta_{\tilde f}:=\tilde f-id_{\R^d}\in C^0(\T^d,\R^d).
\end{equation}
Observe this function can be naturally considered as a cocycle over
$f:=\pi\tilde f$, so we will use the usual notation
\begin{equation}
  \label{eq:Delta-cocycle-property}
  \Delta_{\tilde f}^{(n)}:= \Delta_{\tilde f^n} =
  \mc{S}_f^n\Delta_{\tilde f}, 
  \quad\forall n\in\Z,
\end{equation}
where $\mc{S}_f^n$ denotes the \emph{Birkhoff sum,} \ie given any
$n\in\Z$ and $\phi\colon\T^d\to\R$, $\mc{S}_f^n(\phi)$ is given by
\begin{equation}
  \mc{S}_f^n\phi :=
  \begin{cases}
    \sum_{j=0}^{n-1}\phi\circ f^j,& \text{if } n\geq 1;\\
    0, & \text{if } n=0;\\
    -\sum_{j=1}^{-n}\phi\circ f^{-j},& \text{if } n<0.
  \end{cases}
\end{equation}

The map $\R^d\ni\alpha\mapsto T_\alpha\in\Thomeo(\T^d)$ defines an
injective group homomorphism, and hence, $\R^d$ naturally acts on
$\Thomeo(\T^d)$ by conjugacy. However, since every element of
$\Thomeo(\T^d)$ commutes with $T_{\boldsymbol{p}}$, for all
$\boldsymbol{p}\in\Z^d$, we conclude $\T^d$ itself acts on
$\Thomeo(\T^d)$ by conjugacy, \ie the map
$\Ad\colon\T^d\times\Thomeo(\T^d)\to\Thomeo(\T^d)$ given by
\begin{equation}
  \label{eq:Ad-Td-definition}
  \Ad_t(\tilde f):=T_{\tilde t}^{-1}\circ \tilde f\circ T_{\tilde t},
  \quad\forall(t,\tilde f)\in\T^d\times\Thomeo(\T^d),\ \forall
  \tilde t\in\pi^{-1}(t),
\end{equation}
is well-defined.

\subsection{Invariant measures}
\label{sec:inv-mes-erg-dev-lemm}

Given any topological space $M$, we shall write $\M(M)$ for the space
of Borel probability measures on $M$. We say $\mu\in\M(M)$ is a
\emph{topological measure} when $\mu(U)>0$ for every non-empty open
subset $U\subset M$.

Given any $f\in\Homeo{}(M)$, the space of $f$-invariant probability
measures will be denoted by
$\M(f):=\{\nu\in\M(M) : f_\star\nu = \nu\circ f^{-1}=\nu\}$.

\subsection{Rotation set and rotation vectors}
\label{sec:rotation-set}

Let $f\in\Homeo0(\T^d)$ denote an arbitrary homeomorphism and
$\tilde f\in\Thomeo(\T^d)$ be a lift of $f$. We define the
\emph{rotation set of} $\tilde f$ by
\begin{equation}
  \label{eq:rot-set-def}
  \rho(\tilde f):= \bigcap_{m\geq 0} \overline{\bigcup_{n\geq
      m}\left\{\frac{\Delta_{\tilde f^n}(z)}{n} : z\in\R^d\right\}}.  
\end{equation}
It can be easily shown that $\rho(\tilde f)$ is non-empty, compact and
connected. We say that $f$ is a \emph{pseudo-rotation} when
$\rho(\tilde f)$ is just a point. Notice that whenever
$\tilde f_1,\tilde f_2\in\Thomeo(\T^d)$ are such that
$\pi\tilde f_1=\pi\tilde f_2$, then
$\pi(\rho(\tilde f_1))=\pi(\rho(\tilde f_2))$. Thus, given any
$f\in\Homeo0(\T^d)$, we can just define
\begin{equation}
  \label{eq:rot-set-vect-no-lift}
  \rho(f):=\pi(\rho(\tilde f))\subset\T^d,
\end{equation}
where $\tilde f\in\Thomeo(\T^d)$ denotes any lift of $f$. In
particular, the fact of being a pseudo-rotation depends just on $f$
and not on the chosen lift.

By \eqref{eq:Delta-cocycle-property}, the rotation set is formed by
accumulation points of Birkhoff averages of the displacement function,
so given any $\mu\in\M(\pi\tilde f)$ we can define the \emph{rotation
  vector of $\mu$} by
\begin{equation}
  \label{eq:rot-vector-for-measure}
  \rho_\mu(\tilde f) := \int_{\T^d}\Delta_{\tilde f}\dd\mu,
\end{equation}
and it clearly holds $\rho_\mu(\tilde f)\in\rho(\tilde f)$, whenever
$\mu$ is ergodic.

When $d=1$, by classical Poincaré theory of circle homeomorphisms we
know that $\rho(\tilde f)$ reduces to a point. This does not hold in
higher dimensions, but when $d=2$, after Misiurewicz and Ziemian
\cite{MisiurewiczZiemian} we know that $\rho(\tilde f)$ is not just
connected but also convex. In fact, in the two-dimensional case it
holds
\begin{equation}
  \label{eq:rot-set-vs-inv-meas}
  \rho(\tilde f) = \left\{\rho_{\mu}(\tilde f) : \mu\in
    \M\big(\pi\tilde f\big)\right\}, \quad\forall \tilde
  f\in\Thomeo(\T^2).
\end{equation}

\subsection{Rational rotation vectors and periodic points}
\label{sec:rat-rot-vect-per-points}

Let $\tilde f\in\Thomeo(\T^2)$ be a lift of
$f:=\pi\tilde f\in\Homeo0(\T^2)$. For any $z\in\Per(f)$, there is
$(p_1,p_2,q)\in\Z^2\times\N$ such that
$\tilde f^q(\tilde z)=T_{(p_1,p_2)}(\tilde z)$, for every
$\tilde z\in\pi^{-1}(z)$. Thus, the point
$(p_1/q,p_2/q)\in\rho(\tilde f)\cap\Q^2$, and in this case one says
the periodic point $z$ realizes the rational rotation vector
$(p_1/q,p_2/q)$.

The following result due to Handel asserts that the rotation set of a
periodic point free homeomorphism has empty interior:
\begin{theorem}[Handel \cite{HandelPerPointFree}]
  \label{thm:Handel-ppfree-rho-empty-int}
  Let $f\in\Homeo0(\T^2)$ be a periodic point free homeomorphism and
  $\tilde f\in\Thomeo(\T^2)$ be a lift of $f$. Then, there exists
  $v\in\Ss^1$ and $\alpha\in\R$ so that
  \begin{displaymath}
    \frac{\scprod{\tilde f^n(z)-z}{v}}{n}\to\alpha,\quad \text{as }
    n\to\infty,
  \end{displaymath}
  where the convergence is uniform in $z\in\R^2$. In other words, it
  holds
  \begin{equation}
    \label{eq:rot-set-in-line}
    \rho(\tilde f)\subset\ell_\alpha^v,
  \end{equation}
  where the straight line $\ell_\alpha^v$ is given by
  \eqref{eq:lrv-straight-line}.
\end{theorem}

In general, not every rational point in $\rho(\tilde f)$ is realized
by a periodic orbit of $f$. However, in the non-wandering case the
following holds:
\begin{theorem}
  \label{thm:non-rational-point-ppf-nw-homeo}
  If $f\in\Homeo0(\T^2)$ is non-wandering and periodic point free,
  then
  \begin{equation}
    \label{eq:ppf-homeo-no-rational-point}
    \rho(f)\cap\Q^2/\Z^2=\emptyset.
  \end{equation}
\end{theorem}

\begin{proof}
  Let $\tilde f\in\Thomeo(\T^2)$ be a lift of $f$. By
  Theorem~\ref{thm:Handel-ppfree-rho-empty-int}, $\rho(\tilde f)$ has
  empty interior. So, $\rho(\tilde f)$ is either a point or line
  segment.

  By a result Misiurwicz and Ziemian \cite{MisiurewiczZiemian} we know
  that every rational extreme point of the rotation set is realized by
  a periodic orbit. So, if $f$ is a periodic point free
  pseudo-rotation, then condition
  \eqref{eq:ppf-homeo-no-rational-point} necessarily holds.

  Hence, the only remaining case to consider is when $\rho(\tilde f)$
  is a non-degenerate line segment. In such a case, since $f$ is
  non-wandering, every rational point of $\rho(\tilde f)$ would be
  also realized by a periodic orbit (see
  \cite{JonkerZhangTorHomRot,KocKorFreeCurves,DavalosTorusHomeoRotInt}
  for details), and thus \eqref{eq:ppf-homeo-no-rational-point} holds,
  too.
\end{proof}

\subsection{Directional rotational deviations}
\label{sec:rotat-deviations}

Let $f\in\Homeo0(\T^2)$ and $\tilde f\colon\R^2\carr$ be a lift of
$f$. If the rotation set $\rho(\tilde f)$ has empty interior, then
there exists $\alpha\in\R$ and $v\in\Ss^1$ such that
\begin{equation}
  \label{eq:rot-set-in-line}
  \rho(\tilde f)\subset\ell_\alpha^v=\ell_{-\alpha}^{-v},
\end{equation}
where $\ell_\alpha^v$ is the straight line given by
\eqref{eq:lrv-straight-line}.

In such a case we say that a point $z_0\in\T^d$ exhibits \emph{bounded
  $v$-deviations} when there exists a real constant $M=M(z_0,f)>0$
such that
\begin{equation}
  \label{eq:bound-v-deviations}
  \scprod{\Delta_{\tilde f}^{(n)}(z_0)}{v} - n\alpha \leq M, \quad\forall
  n\in\Z. 
\end{equation}

Moreover, we say that $f$ exhibits \emph{uniformly bounded
  $v$-deviations} when there exists $M=M(f)>0$ such that
\begin{equation}
  \label{eq:unif-bound-v-deviations}
  \scprod{\Delta_{\tilde f}^{(n)}(z)}{v}- n\alpha \leq M, \quad\forall
  z\in\T^d,\ \forall n\in\Z.
\end{equation}

\begin{remark}
  \label{rem:line-v-minus-v-symmetry}
  Notice that the lines $\ell_\alpha^v=\alpha v+\R v^\perp$ and
  $\ell_{-\alpha}^{-v}=(-\alpha)(-v) + \R (-v)^\perp$ coincide as
  subsets of $\R^2$. However \emph{à priori} there is no obvious
  relation between $v$-deviation and $(-v)$-deviation.
\end{remark}

\begin{remark}
  \label{rem:rot-dev-torus-homeo-not-lift}
  Once again, notice that this concept of rotational deviation does
  just depend on the torus homeomorphism and not on the chosen lift.
\end{remark}

\begin{remark}
  \label{rem:no-unif-rot-dev}
  Let us make a final comment about the negation of the above concept:
  we will say that $f$ \emph{does not exhibit uniformly bounded
    $v$-deviations} when for every $M>0$, there exist $z\in\T^2$,
  $n\in\Z$ and $\rho\in\rho(\tilde f)$ such that
  \begin{displaymath}
    \scprod{\Delta_{\tilde f}^{(n)}(z)-n\rho}{v}>M.
  \end{displaymath}
  That means that in such a case we are also considering the
  possibility that there exists no $\alpha\in\R$ such that
  $\rho(\tilde f)\subset\ell_\alpha^v$.
\end{remark}

\subsection{Annular and strictly toral dynamics}
\label{sec:annul-strictly-toral}

Here we recall the concepts of \emph{annular} and \emph{strictly
  toral} homeomorphisms that have been introduced by Koropecki and Tal
in \cite{KoroTalStricTor}.

To do this, let $f\in\Homeo0(\T^2)$ denote an arbitrary homeomorphism.

We say that $f$ is \emph{annular} when there exists a lift
$\tilde f\colon\R^2\carr$, $p/q\in\Q$ and $v\in\Ss^1$ with rational
slope such that $\rho(\tilde f)\subset\ell_{p/q}^v$ and $f$ exhibits
uniformly bounded $v$-deviations.\footnote{Our definition of annular
  map is slightly more general than the one given in
  \cite{KoroTalStricTor}. In fact, they just consider the case
  $p/q=0$.}

On the other hand, $f$ is said to be \emph{strictly toral} when $f$ is
not annular and $\Fix(f^k)$ is not fully essential, for any
$k\in\Z\setminus\{0\}$. 

In order to state the main properties of strictly toral
homeomorphisms, we first need to introduce some notations: for any
$x\in\T^2$ and any $r>0$ let us consider the set
\begin{equation}
  \label{eq:Ur-set-def}
  U_r(x,f):=\cc\bigg(\bigcup_{n\in\Z} f^n(B_r(x)),x\bigg). 
\end{equation}
Then, a point $x\in\T^2$ is said to be \emph{inessential for $f$} if
there is some $r>0$ such that $U_r(x,f)$ is inessential; otherwise $x$
is said to be \emph{essential for $f$.}

As a consequence of Theorem~\ref{thm:Brouwer}, we get the following
\begin{proposition}
  \label{pro:ppf-nonwand-every-point-ess}
  If $f$ is periodic point free and non-wandering, then every point of
  $\T^2$ is essential for $f$.
\end{proposition}

\begin{proof}
  Let us suppose there is an inessential point $x\in\T^2$, \ie there
  exists $r>0$ such that $U_r(x,f)$ is inessential. Since, $U_r(x,f)$
  is a connected component of an $f$-invariant set, and $f$ is
  non-wandering, there is a positive integer $n_0\in\N$ such that
  $f^{n_0}\big(U_r(x,f)\big)=U_r(x,f)$.

  On the other hand, since $U_r(x,f)$ is inessential, its
  \emph{filling} (\ie the union of $U_r(x,f)$ with all the inessential
  connected components of $\T^2\setminus U_r(x,f)$), which will be
  denoted by $U_F$, is a open topological disk; and it can be easily
  shown that $U_F$ is $f^{n_0}$-invariant itself.

  By Theorem~\ref{thm:Brouwer}, $f^{n_0}$ has a fixed point in $U_F$,
  contradicting the hypothesis that $f$ is periodic point free.
\end{proof}

We shall need the following result about strictly toral dynamics
\cite[Proposition 1.4]{KoroTalStricTor}:
\begin{proposition}
  \label{pro:stricly-toral-dynamics}
  If $f$ is strictly toral, then the set $U_r(x,f)$ is fully
  essential, for any essential point $x\in\T^2$ and every $r>0$.
\end{proposition}

\section{The $\tilde\rho$-centralized skew-product}
\label{sec:induced-skew-prod}

Given a lift $\tilde f\in\Thomeo(\T^2)$ of a torus homeomorphism
$f:=\pi\tilde f$ and any vector $\tilde\rho\in\rho(\tilde f)$, we will
define the \emph{$\tilde\rho$-centralized skew-product} induced by
$\tilde f$ which shall play a key role in this work.

To do that, we first define the map $H\colon\T^2\to\Thomeo(\T^2)$ by
\begin{equation}
  \label{eq:Ht-definition}
  H_t:=\Ad_t\left(T_{\tilde\rho}^{-1}\circ\tilde f\right),
  \quad\forall t\in\T^2,
\end{equation}
where $\Ad$ denotes the $\T^2$-action given
by~\eqref{eq:Ad-Td-definition}.

Considering $H$ as a cocycle over the torus translation
$T_\rho\colon\T^2\carr$, where $\rho:=\pi(\tilde\rho)$, one defines
the \emph{$\tilde\rho$-centralized skew-product} as the skew-product
homeomorphism $F\colon\T^2\times\R^2\carr$ given by
\begin{displaymath}
  F(t,z):=\big(T_\rho(t),H_t(z)\big),\quad\forall
  (t,z)\in\T^2\times\R^2.
\end{displaymath}
One can easily show that
\begin{equation}
  \label{eq:F-in-coordinates}
  F(t,z)=\left(t+\rho, z + \Delta_{\tilde f}\big(t+\pi(z)\big) -
    \tilde\rho\right), \quad\forall (t,z)\in\T^2\times\R^2,
\end{equation}
where $\Delta_{\tilde f}\in C^0(\T^2,\R^2)$ is the displacement
function given by \eqref{eq:Delta-cocycle}. We will use the following
classical notation for cocycles: given $n\in\Z$ and $t\in\T^2$, we
write
\begin{displaymath}
  H_t^{(n)}:=
  \begin{cases}
    id_{\T^2},& \text{if } n=0;\\
    H_{t+(n-1)\rho}\circ H_{t+(n-2)\rho}\circ\cdots \circ H_t,&
    \text{if } n>0;\\
    H_{t+n\rho}^{-1}\circ\cdots\circ H_{t-2\rho}^{-1}\circ
    H_{t-\rho}^{-1},& \text{if } n<0.
  \end{cases}
\end{displaymath}
Using such a notation, it holds
$F^n(t,z):=\big(T_\rho^n(t),H_t^{(n)}(z)\big)$, for all
$(t,z)\in\T^2\times\R^2$ and every $n\in\Z$.

This skew-product $F$ will play a fundamental role in our analysis of
rotational deviations and the following simple formula for iterates of
$F$ represents the main reason:
\begin{equation}
  \label{eq:iter-F}
  \begin{split}
    F^n(t,z)&=\left(T_\rho^n(t),H_{t+(n-1)\rho}\bigg(H_{t+(n-2)\rho}\Big(\cdots
      \Big(H_t(z)\big)\Big)\bigg)\right) \\
    &= \left(t+n\rho,\Ad_{t+(n-1)\rho}\big(T_{\tilde\rho}^{-1}\circ
      \tilde f\big)\circ \cdots\circ \Ad_{t}
      \big(T_{\tilde\rho}^{-1}\circ\tilde f\big)(z)\right)\\
    &=\left(t+n\rho,\Ad_t\big(T_{\tilde\rho}^{-n}\circ\tilde
      f^n\big)(z)\right),
  \end{split}
\end{equation}
for every $(t,z)\in\T^2\times\R^2$ and every $n\in\N$.

Notice that assuming inclusion \eqref{eq:rot-set-in-line}, from
\eqref{eq:iter-F} it easily follows that a point $z\in\R^2$ exhibits
bounded $v$-deviations (as in \eqref{eq:bound-v-deviations}) if and
only if
\begin{equation}
  \label{eq:bound-v-deviations-skew-prod}
  \scprod{H_0^{(n)}(z) - z}{v}\leq M, \quad\forall n\in\Z, 
\end{equation}
where $M=M(z,f)$ is the constant given in
\eqref{eq:bound-v-deviations}.

\subsection{Fibered stable sets at infinity}
\label{sec:fibered-stable-sets}

Continuing with the notation we have introduced at the beginning of
\S~\ref{sec:induced-skew-prod}, let $\wh{\R^2}$ denote the one-point
compactification of $\R^2$, and $\wh F\colon\T^2\times\wh{\R^2}\carr$
be the unique continuous extension of $F$, that clearly satisfies
$\wh F(t,\infty):=(T_\rho(t),\infty)$, for every $t\in\T^2$.

Hence, for every $r\in\R$ and each $t\in\T^2$ we define the
\emph{fibered $(r,v)$-stable set at infinity} of $F$ by
\begin{equation}
  \label{eq:Lambda-hat-r-v-definition}
  \wh{\Lambda_{r}^v}\big(\tilde f,
  t\big):=\cc\left(\{t\}\times\wh{\R^2}\cap
    \I_{\wh{F}}\big(\T^2\times\wh{\bb{H}_r^v}\big), (t,\infty) \right),
\end{equation}
where $\bb{H}_r^v$ is the semi-plane given by
\eqref{eq:Hvr-half-spaces-def}, $\wh{\bb{H}^v_r}$ denotes its closure
in $\wh{\R^2}$ and $\I(\cdot)$ is the maximal invariant set given by
\eqref{eq:max-inv-set-def}.

We also define
\begin{equation}
  \label{eq:Lambda-r-v-definition}
  \Lambda_r^v\big(\tilde f, t\big):=\mathrm{pr}_2\left(
    \wh{\Lambda_r^v}\big(\tilde f, t\big)\setminus
    \{(t,\infty)\}\right)\subset\R^2, \quad\forall t\in\T^2, 
\end{equation}
where $\mathrm{pr}_2\colon\T^2\times\R^2\to\R^2$ denotes the
projection on the second coordinate. Finally, we define the
\emph{$(r,v)$-stable set at infinity} by
\begin{equation}
  \label{eq:Lambda-r-v-union-definition}
  \Lambda_r^v\big(\tilde f\big):=\bigcup_{t\in\T^2}
  \{t\}\times\Lambda_r^v\big(\tilde f,t\big)\subset\T^2\times\R^2.
\end{equation}

For the sake of simplicity, if there is no risk of confusion we shall
just write $\wh{\Lambda_r^v}(t)$, $\Lambda_r^v(t)$ and $\Lambda_r^v$
instead of $\wh{\Lambda_r^v}\big(\tilde f, t\big)$,
$\Lambda_r^v\big(\tilde f,t\big)$ and $\Lambda_r^v\big(\tilde f\big)$,
respectively.

By the very definitions, stable sets at infinity are closed and
$F$-invariant, but at first glance they might be empty. In
Theorem~\ref{thm:Lambda-non-empty} we will prove this is never the
case and we shall use them to get some results about rotational
$v$-deviations for the original homeomorphism $f$.

Our first application of these stable sets at infinity is the
following result which concerns the symmetry of boundedness of $v$-
and $(-v)$-deviations and is just a simple consequence of the
$F$-invariance:

\begin{theorem}
  \label{thm:v-bound-dev-iff--v-bound-dev}
  Assuming inclusion \eqref{eq:rot-set-in-line}, the homeomorphism $f$
  exhibits uniformly bounded $v$-deviations if and only if it exhibits
  uniformly bounded $(-v)$-deviations. More precisely, it holds
  \begin{displaymath}
    \begin{split}
      \sup_{n\in\Z}\sup_{z\in\T^2} \scprod{\Delta_{\tilde
          f}^{(n)}(z)-n\tilde \rho}{v} -\sqrt{2}\leq
      \sup_{n\in\Z}\sup_{z\in\T^2}& \scprod{\Delta_{\tilde
          f}^{(n)}(z)-n\tilde\rho}{-v} \\
      &\leq \sup_{n\in\Z}\sup_{z\in\T^2} \scprod{\Delta_{\tilde
          f}^{(n)}(z)-n\tilde \rho}{v} +\sqrt{2},
    \end{split}
  \end{displaymath}
  for $\tilde\rho\in\rho(\tilde f)$.
\end{theorem}

\begin{proof}
  By Remark \ref{rem:line-v-minus-v-symmetry}, the statement is
  completely symmetric with respect to $v$ and $-v$. Thus, let us
  assume $f$ exhibits uniformly bounded $v$-deviations and let us
  prove uniformly boundedness for $(-v)$-deviations. So let us define
  \begin{displaymath}
    M:=\sup_{n\in\Z}\sup_{z\in\T^2} \scprod{\Delta_{\tilde
        f}^{(n)}(z)-n\tilde\rho}{v}\in\R,
  \end{displaymath}
  and notice that $M\geq 0$. Then, by
  \eqref{eq:bound-v-deviations-skew-prod} we know
  \begin{equation}
    \label{eq:bound-H0z-from-below}
    \scprod{H_0^{(n)}(z) - z}{v}\leq M,\quad\forall
    z\in\R^2,\ \forall n\in\Z,
  \end{equation}
  and consequently $\bb{H}_{M-r}^{-v}\subset\Lambda_{-r}^{-v}(0)$.

  For an arbitrary $t\in\T^2$, notice that
  \begin{equation}
    \label{eq:Ht-displacement-function-any-t}
    H_t^{(n)}(z) - z = \Delta_{\tilde f}^{(n)}\big(\pi(z)+t\big) -
    n\tilde\rho, \quad\forall (t,z)\in\T^2\times\R^2,\ \forall
    n\in\Z. 
  \end{equation}

  Then, putting together \eqref{eq:bound-H0z-from-below} and
  \eqref{eq:Ht-displacement-function-any-t} we conclude that
  \begin{displaymath}
    \scprod{H_t^{(n)}(z) - z}{v}\leq M,\quad\forall
    (t,z)\in\T\times\R^2,\ \forall n\in\Z,
  \end{displaymath}
  and hence, $\bb{H}_{M-r}^{-v}\subset\Lambda_{-r}^{-v}(t)$, for any
  $t\in\T^2$.

  Thus, we have
  \begin{displaymath}
    \bb{H}_M^{-v}\subset\Lambda_0^{-v}(t)\subset\bb{H}_0^{-v},
    \quad\forall t\in\T^2,
  \end{displaymath}
  and since $\Lambda_0^{-v}\subset\T^2\times\bb{H}_0^{-v}$ is an
  $F$-invariant set, this implies
  \begin{equation}
    \label{eq:F-invariant-sets-disjoint}
    F^n\big(\T^2\times\bb{H}_\varepsilon^v)\cap\T^2\times\bb{H}_{M}^{-v}
    = \emptyset, \quad\forall n\in\Z,
  \end{equation}
  and for any $\varepsilon>0$.

  In particular, if $D\subset\R^2$ denotes a squared fundamental
  domain for the covering map $\pi\colon\R^2\to\T^2$ such that
  $D\subset\bb{H}_\epsilon^v$, by \eqref{eq:F-in-coordinates} and
  \eqref{eq:F-invariant-sets-disjoint} it follows that
  \begin{equation}
    \label{eq:-v-dev-on-fund-domain}
    \scprod{\Delta_{\tilde f^n}(z) - n\tilde\rho}{-v} < M +
    \varepsilon + \mathrm{diam}(D) =
    M+\varepsilon+\sqrt{2},\quad\forall z\in D. 
  \end{equation}
  Since the displacement function $\Delta_{\tilde f^n}$ is
  $\Z^2$-periodic and \eqref{eq:-v-dev-on-fund-domain} holds for any
  $\varepsilon>0$, we conclude that
  \begin{displaymath}
    \sup_{n\in\Z}\sup_{z\in\T^2}\scprod{\Delta_{\tilde
        f^n}(z)-n\tilde\rho}{-v}\leq \sup_{n\in\Z}\sup_{z\in\T^2}
    \scprod{\Delta_{\tilde f^n}(z)-n\tilde\rho}{v}+ \sqrt{2}.
  \end{displaymath}
  In particular, $f$ exhibits uniformly bounded $(-v)$-deviations.
\end{proof}

For the sake of concreteness, we explicitly state the following
corollary which is a straightforward consequence of
Theorem~\ref{thm:v-bound-dev-iff--v-bound-dev}:
\begin{corollary}
  \label{cor:unif-bound-dev-iff-Hm-in-Lambda}
  The following assertions are all equivalent:
  \begin{enumerate}[(i)]
  \item $f$ exhibits uniformly bounded $v$-deviations;
  \item $f$ exhibits uniformly bounded $(-v)$-deviations;
  \item there exists $M>0$ such that
    \begin{displaymath}
      \bb{H}_{r+M}^v\subset\Lambda_r^v(t), \quad\forall r\in\R,\
      \forall t\in\T^2;
    \end{displaymath}
  \item there exists $M>0$ such that
    \begin{displaymath}
      \bb{H}_{r+M}^{-v}\subset\Lambda_r^{-v}(t), \quad\forall r\in\R,\
      \forall t\in\T^2.
    \end{displaymath}
  \end{enumerate}
\end{corollary}

Our next result describes some elementary equivariant properties of
$(r,v)$-stable sets at infinity:
\begin{proposition}
  \label{pro:Lambda-equivariant-properties}
  For each $t\in\T^2$ and any $r\in\R$, the following properties hold:
  \begin{enumerate}[(i)]
  \item \label{eq:Lambda-r-monoton-inclusion}
    $\Lambda_r^v(t)\subset\Lambda_{s}^v(t)$, for every $s<r$;
  \item \label{eq:Lambda-r-monoton-limit}
    \begin{displaymath}
      \Lambda_r^v(t)=\bigcap_{s<r}\Lambda_s^v(t);
    \end{displaymath}
  \item \label{eq:Lambda-t-conjug}
    \begin{displaymath}
      \Lambda_{r+\langle\tilde t,v\rangle}^v\big(t'-\pi\big(\tilde
      t\big)\big) = T_{\tilde t}\big(\Lambda_r^v(t')\big),
      \quad\forall\tilde t\in\R^2,\ \forall t'\in\T^2;
    \end{displaymath}
  \item \label{eq:Lambda-Z2-translations}
    \begin{displaymath}
      T_{\boldsymbol{p}}\big(\Lambda_r^v(t)\big) =
      \Lambda^v_{r+\langle\boldsymbol{p},v\rangle}(t),\quad\forall
      \boldsymbol{p}\in\Z^2. 
    \end{displaymath}  
  \end{enumerate}
\end{proposition}

\begin{proof}
  Inclusion \eqref{eq:Lambda-r-monoton-inclusion} trivially follows
  from the inclusion $\bb{H}_r^v\subset\bb{H}_{s}^v$, for $s<r$.

  To show \eqref{eq:Lambda-r-monoton-limit}, let
  $\wh{\Lambda_s^v(t)}=\Lambda_s^v(t)\sqcup\{\infty\}$ denote the
  closure of $\Lambda_s^v(t)$ inside the one-point compactification
  $\wh{\R^2}$. Since $\wh{\Lambda_s^v(t)}$ is compact and connected
  for every $s\in\R$, by \eqref{eq:Lambda-r-monoton-inclusion} we
  conclude
  \begin{displaymath}
    \bigcap_{s<r}\wh{\Lambda_s^v(t)}
  \end{displaymath}
  is compact and connected itself and contains
  $\wh{\Lambda_r^v(t)}$. On the other hand, it clearly holds
  $\bigcap_{s<r}\Lambda_s^v(t)\subset\I_F(\T^2\times\bb{H}_r^v)$ and
  thus, $\bigcap_{s<r}\Lambda_s^v(t)=\Lambda_r^v(t)$.

  To prove \eqref{eq:Lambda-t-conjug}, taking into account
  \eqref{eq:Ad-Td-definition} and \eqref{eq:Ht-definition}, we get
  \begin{displaymath}
    \begin{split}
      H_{t'-\pi(\tilde t)}^{(n)}(z) &= \Ad_{t'-\pi(\tilde t)}
      \big(T^{-n}_{\tilde \rho}\circ\tilde f^n\big)(z) \\
      &= T_{\tilde t}\circ \Ad_{t'}\big(R^{-n}_\rho\circ\tilde
      f^n\big)\circ T_{\tilde t}^{-1} (z) \\
      &= H^{(n)}_{t'}(z-\tilde t) + \tilde t,
    \end{split}
  \end{displaymath}
  for every $z\in\R^2$ and every $n\in\Z$. Hence, for any
  $z\in\Lambda^v_{r+\scprod{\tilde t}{v}}\big(t'-\pi(\tilde t)\big)$,
  it holds
  \begin{displaymath}
    \begin{split}
      r&\leq \scprod{H_{t'-\pi(\tilde t)}^{(n)}(z)}{v} -
      \scprod{\tilde t}{v}\\
      &= \scprod{H^{(n)}_{t'}(z-\tilde t) + \tilde t}{v} -
      \scprod{\tilde t}{v} = \scprod{H_{t'}^{(n)}(z-\tilde t)}{v},
    \end{split}
  \end{displaymath}
  for any $n\in\Z$. This implies
  $z\in T_{\tilde t}\big(\Lambda_r^v(t')\big)$, and so,
  $\Lambda^v_{r+\scprod{\tilde t}{v}}\big(t'-\pi(\tilde t)\big)\subset
  T_{\tilde t}\big(\Lambda_r^v(t')\big)$. The other inclusion is
  symmetric.

  Finally, relation \eqref{eq:Lambda-Z2-translations} is just a
  particular case of \eqref{eq:Lambda-t-conjug}.
\end{proof}

Now we can show the main result of this section:
\begin{theorem}
  \label{thm:Lambda-non-empty}
  Assuming there exist $\alpha\in\R$ and $v\in\Ss^1$ such that
  condition \eqref{eq:rot-set-in-line} holds, the fibered
  $(r,v)$-stable set at infinity $\Lambda_r^v(t)$ is non-empty, for
  all $r\in\R$ and every $t\in\T^2$.
\end{theorem}

\begin{proof}[Proof of Theorem~\ref{thm:Lambda-non-empty}]
  Our strategy to show $\Lambda_r^v(t)$ is non-empty is mainly
  inspired by some ideas due to Birkhoff~\cite{BirkhoffNouvelles}.

  By Corollary~\ref{cor:unif-bound-dev-iff-Hm-in-Lambda}, we know that
  $\Lambda_r^v(t)\neq\emptyset$ whenever $f$ exhibits uniformly
  bounded $(-v)$-deviations.

  Hence we can suppose this is not the case, so either
  \begin{equation}
    \label{eq:unbounded-deviations-pos-times}
    \sup_{n\geq 0}\sup_{z\in\T^2} \scprod{\Delta_{\tilde f}^{(n)}(z)
      -n\tilde\rho}{-v}=+\infty,
  \end{equation}
  or
  \begin{equation}
    \label{eq:unbounded-deviations-neg-times}
    \sup_{n\leq 0}\sup_{z\in\T^2} \scprod{\Delta_{\tilde f}^{(n)}(z)
      -n\tilde\rho}{-v}=+\infty.
  \end{equation}

  For the sake of concreteness, let us suppose
  \eqref{eq:unbounded-deviations-neg-times} holds. Then, consider the
  set
  \begin{displaymath}
    \wh{B^+}:=\bigcap_{j=0}^{+\infty}
    \wh{F}^{-j}\Big(\T^2\times\wh{\bb{H}_r^v}\Big)\subset
    \T^2\times\wh{\R^2}. 
  \end{displaymath}

  We will show that $\wh{B^+}$ exhibits unbounded connected components
  along the fibers. More precisely, for each $t\in\T^2$ we define
  \begin{equation}
    \label{eq:B-wh-plus-t}
    B^+(t):=\mathrm{pr}_2\left(\cc\Big(\wh{B^+}\cap
      \{t\}\times\wh{\R^2},(t,\infty)\Big)\right)
    \setminus\{(t,\infty)\} \subset\R^2, 
  \end{equation}
  where $\mathrm{pr}_2\colon\T^2\times\R^2\to\R^2$ denotes the
  projection on the second factor, and we shall show that $B^+(t)$ is
  non-empty, for all $t$.

  Since we are assuming \eqref{eq:unbounded-deviations-neg-times}
  holds, for every $a>0$ we can define the following natural number:
  \begin{equation}
    \label{eq:n-a-t-def}
    n(a):=\min\Big\{n\in\N :
    F^{-n}\big(\T^2\times\ell_{a+r}^v\big)\cap
    \big(\T^2\times\bb{H}_{-r}^{-v}\big) \neq\emptyset\Big\}.
  \end{equation}
  Then, there exists a point $t_a\in\T^2$ such that
  \begin{displaymath}
    F^{-n(a)}\big(\{t_a+n(a)\rho\}\times\bb{H}_{r+a}^v\big)\cap 
    \{t_a\}\times\bb{H}_{-r}^{-v}\neq\emptyset.
  \end{displaymath}
  So, we can find a simple continuous arc
  $\gamma_{a}\colon [0,1]\to\wh{\R^2}$ such that
  \begin{equation}
    \label{eq:gamma-t-a-def}
    \begin{split}
      &\gamma_{a}(0)\in\ell_r^v=\partial\bb{H}_r^v,\\
      &\gamma_{a}(1)=\infty,\\
      &\gamma_{a}[0,1)\subset\mathrm{pr}_2
      \Big(F^{-n(a)}\big(\{t+n(a)\rho\}\times\bb{H}_{r+a}^v\big)\Big)
      \cap\bb{H}_r^v,
    \end{split}
  \end{equation}
  and a lattice point $\boldsymbol{p}_{a}\in\Z^2$ such that
  \begin{equation}
    \label{eq:gamma-t-a-bringing-to-I2}
    \begin{split}
      &\scprod{\boldsymbol{p}_{a}}{v}\geq 0,\\
      &\norm{\gamma_{a}(0)+\boldsymbol{p}_{a}-rv} \leq\sqrt{2}.
    \end{split}
  \end{equation}
  Let us define the set
  \begin{equation}
    \label{eq:Gamma-a-t-def}
    \Gamma_{a}:=\left\{\gamma_{a}(s) +
      \boldsymbol{p}_{a} : s\in[0,1)\right\}\subset\R^2. 
  \end{equation}

  Putting together \eqref{eq:n-a-t-def} and \eqref{eq:gamma-t-a-def},
  we conclude that
  \begin{displaymath}
    F^j\big(t_a,\Gamma_{a}\big)\subset
    \T^2\times\bb{H}_r^v,\quad\text{for } 0\leq j\leq n(a)-1.
  \end{displaymath}

  Now, since $\wh{\R^2}$ is compact, its space of compact subsets
  $\mc{K}(\wh{\R^2})$ is compact, too. Hence, we can find a strictly
  increasing sequence of natural numbers $(m_j)$, a point
  $t_\infty\in\T^2$ and $\wh{\Gamma}\in\mc{K}(\wh{\R^2})$ such that
  $t_{m_j}\to t_\infty$ and
  \begin{displaymath}
    \wh{\Gamma_{m_j}}\to\wh{\Gamma},\quad\text{as } j\to\infty,
  \end{displaymath}
  where $\wh{\Gamma_{m_j}}$ denotes the closure of $\Gamma_{m_j}$
  inside the space $\wh{\R^2}$ and the convergence is in the Hausdorff
  distance.

  Then observe that $n(m_j)\to+\infty$ as $m_j\to+\infty$. So,
  $\wh{\Gamma}\subset\wh{B^+}$. Then, the set
  $\Gamma:=\wh\Gamma\setminus\{(t_\infty,\infty)\}$ is closed in
  $\R^2$, connected, and by \eqref{eq:gamma-t-a-bringing-to-I2}, is
  non-empty and unbounded. In particular,
  $\Gamma\subset B^+(t_\infty)$ and so, $B^+(t_\infty)$ is non-empty.

  Now, consider an arbitrary point $t\in\T^2$ and let us show that
  $B^+(t)$ is non-empty, too. To do this, let us choose two points
  $\tilde t_\infty\in\pi^{-1}(t_\infty)$ and $\tilde t\in\pi^{-1}(t)$
  so that $\scprod{\tilde t_\infty-\tilde t}{v}\geq 0$. Hence, it
  holds
  \begin{equation}
    \label{eq:T-t-infty-t-B-in-Hvr}
    T_{\tilde t_\infty-\tilde t}\big(\bb{H}_r^v\big)
    \subset\bb{H}_r^v.
  \end{equation}

  On the other hand, invoking \eqref{eq:Ht-definition} one easily see
  that
  \begin{displaymath}
    \begin{split}
      H_t^{(n)}\left(T_{\tilde t_\infty-\tilde
          t}\big(B^+(t_\infty)\big)\right) &=
      \Ad_t(T_{\tilde\rho}^{-n}\circ \tilde f^n) \left(T_{\tilde
          t_\infty-\tilde t}\big(B^+(t_\infty)\big)\right) \\
      & = T_{\tilde t}^{-1}\circ T_{\tilde\rho}^{-n} \circ\tilde
      f^n\circ T_{\tilde t}\circ T_{\tilde t_\infty-\tilde
        t}\big(B^+(t_\infty)\big)\\
      &= T_{\tilde t_\infty-\tilde t}\circ
      \Ad_{t_\infty}(T_{\tilde\rho}^{-n}\circ\tilde
      f^n)\big(B^+(t_\infty)\big) \\
      &\subset T_{\tilde t_\infty-\tilde
        t}(\bb{H}_r^v)\subset\bb{H}_r^v,
    \end{split}
  \end{displaymath}
  for every $n\geq 0$. So, we conclude
  $T_{\tilde t_\infty-\tilde t}(B^+(\tilde t_\infty))\subset B^+(t)$
  and in particular, $B^+(t)$ is non-empty.

  Finally, let us show that $\Lambda^v_r(t)\neq\emptyset$, for every
  $t\in\T^2$. To do that, let us consider the set
  \begin{displaymath}
    \mathscr{B}^+:=\bigcup_{t\in\T^2}\{t\}\times B^+(t)\subset
    \T^2\times\R^2. 
  \end{displaymath} 
  Notice that $F(\mathscr{B}^+)\subset\mathscr{B}^+$ and
  \begin{displaymath}
    \Lambda_r^v=\bigcap_{n\geq 0} F^n(\mathscr{B}^+).
  \end{displaymath}

  So, if we suppose that $\Lambda_r^v=\emptyset$, then there should
  exist $\boldsymbol{p}\in\Z^2$ and $n_0\geq 0$ such that
  $\scprod{\boldsymbol{p}}{v}>0$ and
  \begin{displaymath}
    F^{n_0}(\mathscr{B}^+)\subset \T^2\times
    T_{\boldsymbol{p}}(\bb{H}_r^v) \subsetneq
    \T^2\times\bb{H}_r^v.
  \end{displaymath}
  But since $F$ commutes with the map
  $id\times T_{\boldsymbol{p}}\colon\T^2\times\R^2\carr$, it follows
  that
  \begin{displaymath}
    F^{jn_0}(\mathscr{B}^+)\subset
    \T^2\times T_{\boldsymbol{p}}^j(\bb{H}_r^v),\quad \forall j\in\N,
  \end{displaymath}
  contradicting inclusion \eqref{eq:rot-set-in-line}. So, we have
  showed $\Lambda_r^v$ is non-empty. In particular, there exists
  $t'\in\T^2$ such that $\Lambda_r^v(t')\neq\emptyset$. Invoking the
  very same argument we used to show that for all $t$, $B^+(t)$ was
  non-empty provided $B^+(t_\infty)\neq\emptyset$, one can prove that
  $\Lambda_r^v(t)\neq\emptyset$, for every $t\in\T^2$.
\end{proof}

The last result of this section is an elementary fact about the
topology of fibered invariant sets at infinity under the assumption of
unbounded $v$-deviations:

\begin{proposition}
  \label{pro:Lambda-v-unbounded}
  If $f$ does not exhibits uniformly bounded $v$-deviations, then for
  any $r\in\R$, any $t\in\T^2$ and any $z\in\Lambda_r^v(t)$, it holds
  \begin{equation}
    \label{eq:Lambda-cc-v-unbounded}
    \cc\left(\Lambda_r^v(t),z\right)\nsubset \A_s^v,\quad\forall s>0,
  \end{equation}
  where $\A_s^v$ is the strip given by \eqref{eq:Asv-strip-def}.
\end{proposition}

\begin{proof}
  Let us suppose \eqref{eq:Lambda-cc-v-unbounded} is not true. Then,
  by \eqref{eq:Lambda-t-conjug} of
  Proposition~\ref{pro:Lambda-equivariant-properties}, we know there
  exist $s>0$ and $z\in\Lambda_0^v(0)$ such that
  \begin{displaymath}
    \Lambda_z:=\cc\left(\Lambda_0^v(0),z\right)\subset\A_s^v.
  \end{displaymath}
  Since every connected component of $\Lambda_0^v(0)$ is unbounded in
  $\R^2$, without loss of generality we can assume there exists a
  sequence $(z_n)_{n\geq 0}\subset\Lambda_z$ such that
  \begin{equation}
    \label{eq:z-n-in-Lambda-z-to-+infty}
    \lim_{n\to+\infty}\scprod{z_n}{v^\perp}=+\infty. 
  \end{equation}

  On the other hand, there exists a sequence
  $(\boldsymbol{p}_n)_{n\geq 0}$ in $\Z^2$ such that
  \begin{gather}
    \label{eq:m-n-to-right}
    0\leq \langle \boldsymbol{p}_n,v\rangle \leq 1, \quad\forall
    n\geq 0, \\
    \label{eq:m-n-to-infty}
    \lim_{n\to+\infty}\scprod{\boldsymbol{p}_n}{v^\perp}=-\infty.
  \end{gather}

  By Proposition \ref{pro:Lambda-equivariant-properties} and
  \eqref{eq:m-n-to-right}, for such a sequence it holds
  \begin{equation}
    \label{eq:Tmn-Lambda-z-in-As+1}
    T_{\boldsymbol{p}_n}(\Lambda_z)\subset\Lambda_0^v(0)\cap\A_{s+1}^v,
    \quad\forall n\geq 0.
  \end{equation}

  Then, since $\tilde f$ preserves orientation and $\Lambda_z$ is
  connected and unbounded, as a consequence of \eqref{eq:m-n-to-infty}
  and \eqref{eq:Tmn-Lambda-z-in-As+1} we get that
  $\bigcup_{n\geq 0} T_{\boldsymbol{p}_n}(\Lambda_z)$ disconnects
  $\R^2$, and hence, $\bb{H}^v_{s+1}\subset\Lambda_0^v(0)$. By
  Corollary~\ref{cor:unif-bound-dev-iff-Hm-in-Lambda}, this implies
  $f$ exhibits uniformly bounded $v$-deviations, contradicting our
  hypothesis.
\end{proof}

\section{Directional deviations for periodic point free
  homeomorphisms}
\label{sec:dir-dev-ppf-homeos}

In this section we analyze the topological properties of stable sets
at infinity for non-wandering periodic point free homeomorphisms.

So, let $f\in\Homeo0(\T^2)$ be a non-wandering homeomorphism with
$\Per(f)=\emptyset$ and $\tilde f\colon\R^2\carr$ a lift of $f$. By
Theorem~\ref{thm:Handel-ppfree-rho-empty-int} we know that
$\rho(\tilde f)$ has empty interior. Thus, without any further
assumption, there are some $\alpha\in\R$ and $v\in\Ss^1$ such that
condition \eqref{eq:rot-set-in-line} holds.

Continuing with the notation we introduced in
\S~\ref{sec:induced-skew-prod}, first we will prove a density result:

\begin{theorem}
  \label{thm:translates-of-Lambdas-dense}
  For every $t\in\T^2$ the set
  \begin{displaymath}
    \bigcup_{r\geq 0}\Lambda_{-r}^v(t)
  \end{displaymath}
  is dense in $\R^2$.
\end{theorem}

\begin{proof}
  By Corollary~\ref{cor:unif-bound-dev-iff-Hm-in-Lambda},
  Theorem~\ref{thm:translates-of-Lambdas-dense} clearly holds when $f$
  exhibits uniformly bounded $v$-deviations. So, from now on we shall
  suppose this is not the case.

  Consider the set
  \begin{displaymath}
    \tilde\Lambda:=\overline{\bigcup_{r\geq
        0}\Lambda_{-r}^v\big(\tilde f,0\big)}\subset\R^2 
  \end{displaymath}

  We claim that $\Lambda:=\pi(\tilde\Lambda)\subset\T^2$ is closed,
  $f$-invariant, and it holds $\tilde\Lambda=\pi^{-1}(\Lambda)$.
  
  In order to prove our claim, first observe that $\tilde\Lambda$ is
  $\Z^2$-invariant. In fact, this easily follows from
  \eqref{eq:Lambda-Z2-translations} of
  Proposition~\ref{pro:Lambda-equivariant-properties}, noticing
  \begin{displaymath}
    T_{\boldsymbol{p}}\Big(\Lambda_{-r}^v\big(\tilde f,0\big)\Big) =
    \Lambda_{\langle\boldsymbol{p},v\rangle-r}^v\big(\tilde f,0\big),
    \quad\forall r\geq 0,\ \forall\boldsymbol{p}\in\Z^2.  
  \end{displaymath}
  This immediately implies that $\Lambda$ is closed, too, and
  $\tilde\Lambda = \pi^{-1}(\Lambda)$.

  Secondly, we will prove $\tilde\Lambda$ is $\tilde f$-invariant. To
  do that, observe that for every $r\in\R$ and every point
  $z\in\Lambda_r^v(0)$, it holds
  \begin{displaymath}
    r\leq \scprod{\tilde f^n(z) - n\tilde\rho}{v} = \scprod{\tilde
      f^{n-1}\big(\tilde f(z)\big) - (n-1)\tilde\rho}{v} -
    \scprod{\tilde\rho}{v}, 
  \end{displaymath}
  for every $n\in\Z$.  This implies that
  \begin{displaymath}
    \tilde f(z)\in\Lambda^v_{r+\langle\tilde\rho,v\rangle}\big(\tilde
    f,0\big),
  \end{displaymath}
  and analogously one can show that
  \begin{displaymath}
    \tilde f^{-1}(z)\in\Lambda^v_{r-\langle\tilde\rho,v\rangle}\big(\tilde
    f,0\big).
  \end{displaymath}
  Thus, $\tilde\Lambda$ is totally $\tilde f$-invariant, and we finish
  the proof of our claim.

  Now, let us suppose $\tilde\Lambda\neq\R^2$. This implies that
  $\Lambda\neq\T^2$. First, observe that none connected component of
  $\Lambda$ can be inessential in $\T^2$ (see
  \S\ref{sec:surface-topology} for definitions). In fact, every
  connected component of the pre-image by $\pi\colon\R^2\to\T^2$ of a
  compact inessential set in $\T^2$ is bounded in $\R^2$, and by
  definition, every connected component of $\tilde\Lambda$ is
  unbounded.

  Henceforth, if $A$ denotes an arbitrary connected component of
  $\T^2\setminus\Lambda$, then $A$ is either a topological disk or
  annular (see \S\ref{sec:surface-topology} for definitions). Since
  $f$ is non-wandering and $A$ is a connected component of an open
  $f$-invariant set, there exists $n_0=n_0(A)\in\N$ such that
  $f^{n_0}(A)=A$.

  So, if $A$ were a topological disk, then
  $f^{n_0}\big|_A\colon A\carr$ would be conjugate to a plane
  orientation-preserving non-wandering homeomorphism, and by
  Theorem~\ref{thm:Brouwer}, we know that
  $\Fix\big(f^{n_0}\big|_A\big)\neq\emptyset$, contradicting our
  hypothesis that $f$ is periodic point free.

  Thus, $A$ should be annular, \ie an essential open topological
  annulus. Let $\tilde A$ be any connected component of
  $\pi^{-1}(A)\subset\R^2$. Since $A$ is annular, $\tilde A$ separates
  the plane in two different connected components, such that both of
  them are unbounded, and has a well-defined integral homological
  direction, \ie there exists a rational slope vector
  $v^\prime\in\Ss^1$ (which is unique up to multiplication by $(-1)$)
  and $s=s(\tilde A)>0$ such that
  \begin{displaymath}
    \tilde A\subset\A^{v'}_s. 
  \end{displaymath}

  Then we have to consider two possible cases: when $v$ and $v^\prime$
  are co-linear; and when they are linearly independent in $\R^2$.

  In the first case, let $z\in\Lambda_0^v(0)$ be an arbitrary point,
  $\Lambda_z$ be the connected component of $\Lambda_0^v(0)$
  containing $z$ and $\boldsymbol{p}\in\Z^2$ such that
  $T_{\boldsymbol{p}}(z)\in\bb{H}^v_{-s-1}$. By
  Proposition~\ref{pro:Lambda-equivariant-properties} we know that
  $T_{\boldsymbol{p}}(\Lambda_z)$ is a connected component of
  $\Lambda_{\langle\boldsymbol{p},v\rangle}^v(0)\subset\tilde\Lambda$,
  and since we are assuming $f$ does not exhibit uniformly bounded
  $v$-deviations, we can apply
  Proposition~\ref{pro:Lambda-v-unbounded} to conclude that
  $T_{\boldsymbol{p}}(\Lambda_z)$ is not contained in any
  $v$-strip. Hence, $T_{\boldsymbol{p}}(\Lambda_z)$ intersects both
  semi-spaces $\bb{H}_{s+1}^v$ and $\bb{H}^v_{-s-1}$. Now, since
  $\tilde A$ separates $\bb{H}_{s+1}^v$ and $\bb{H}^v_{-s-1}$, we
  conclude $T_{\boldsymbol{p}}(\Lambda_z)$ should intersect
  $\tilde A$, too. In particular
  $\tilde\Lambda\cap\tilde A\neq\emptyset$ and henceforth,
  $\Lambda\cap A\neq\emptyset$, contradicting our hypothesis that $A$
  was connected component of $\T^2\setminus\Lambda$.

  Let us analyze the second case, \ie when $v$ and $v^\prime$ are
  linear independent in $\R^2$. We know there exists $n_0\geq 1$ such
  that $f^{n_0}(A)=A$. Hence, there exists $\boldsymbol{p}_0\in\Z^2$
  such that $\tilde f^{n_0}(\tilde A)=T_{\boldsymbol{p}_0}(\tilde A)$.

  Let us consider the homeomorphism
  $\tilde g:=T_{\boldsymbol{p}_0}^{-1}\circ\tilde
  f^{n_0}\in\Thomeo(\T^2)$, and notice that $\tilde g$ is a lift of
  $f^{n_0}$ and $\tilde g(\tilde A)=\tilde A$. Then,
  $n_0\rho(\tilde f)-\boldsymbol{p}_0=\rho(\tilde
  g)\subset\ell_0^{v'}$, and since $v$ and $v'$ are not co-linear, we
  conclude that $\tilde f$ is pseudo-rotation, \ie
  $\rho(\tilde f)=\{\tilde\rho\}$.

  Then, let $G\colon\T^2\times\R^2\carr$ be the
  $(n_0\tilde\rho-\boldsymbol{p}_0)$-centralized skew-product induced
  by $\tilde g$, as defined at the beginning of
  \S\ref{sec:induced-skew-prod}. Since $\tilde g(\tilde A)=\tilde A$,
  it easily follows that the open set
  \begin{displaymath}
    \hat A:=\left\{(t,z) : t\in\T^2,\ z\in T_{\boldsymbol{p}-\tilde
        t}(\tilde A),\ \tilde t\in\pi^{-1}(t),\
      \boldsymbol{p}\in\Z^2\right\}\subset\T^2\times\R^2 
  \end{displaymath}
  is $G$-invariant. Now, we are supposing $v$ and $v^\prime$ are not
  co-linear, so we can repeat the argument used in the proof
  Theorem~\ref{thm:Lambda-non-empty} inside the set $\hat A$ to show
  that $\Lambda_r^v\big(\tilde g\big)\cap\hat
  A\neq\emptyset$. Moreover, it can be shown that
  $\Lambda_r^v\big(\tilde g,0\big)\cap\tilde A\neq\emptyset$. Finally
  observe that
  \begin{displaymath}
    \Lambda_r^v(\tilde g,0\big)\subset\Lambda_{r-M}^v\big(\tilde f,0),
  \end{displaymath}
  where $M:=n_0\sup_{z\in\R^2}\abs{\Delta_{\tilde f}(z)}$. In
  particular, we have shown that
  $\tilde\Lambda\cap\tilde A\neq\emptyset$, and thus,
  $\Lambda\cap A\neq\emptyset$, contradicting our hypothesis that $A$
  is a connected component of $\T^2\setminus\Lambda$.
\end{proof}

In Corollary~\ref{cor:unif-bound-dev-iff-Hm-in-Lambda} we have shown
that $f$ exhibits uniformly bounded $(\pm v)$-deviations if and only
if the fibered $(r,v)$-stable sets at infinity contain whole
semi-planes.

Here, we will improve this result for periodic point free systems:

\begin{theorem}
  \label{thm:unbounded-dev-implies-Lambda-empty-interior}
  Let $f\in\Homeo0(\T^2)$ be a periodic point free and non-wandering
  homeomorphism, $\tilde f\colon\R^2\carr$ a lift of $f$, and
  $v\in\Ss^1$ and $\alpha\in\R$ such that inclusion
  \eqref{eq:rot-set-in-line} holds. If the set $\Lambda_r^v(t)$ has
  non-empty interior for some (and hence, any) $t\in\T^2$ and
  $r\in\R$, then there exists $v' \in \Ss^1$ such that $f$ exhibits
  uniformly bounded $v'$-deviations.
\end{theorem}

\begin{proof}
  If $f$ is an annular homeomorphism, (see
  \S\ref{sec:annul-strictly-toral} for the definition), then the
  conclusion automatically holds. So, we can assume $f$ is
  non-annular.

  On the other hand, since $f$ is periodic point free, it clearly
  holds $\Fix(f^k)=\emptyset$, for every $k\in\Z\setminus\{0\}$. So,
  according to the classification given in
  \S~\ref{sec:annul-strictly-toral}, we can assume $f$ is a strictly
  toral homeomorphism,

  Now, since $f$ is non-wandering and periodic point free, by
  Proposition~\ref{pro:ppf-nonwand-every-point-ess} we know that every
  point of $\T^2$ is essential for $f$.
  
  Let us suppose $\Lambda_r^v(0)$ has non-empty interior. Let $x$ be
  a point in interior of $\Lambda_r^v(0)$, and $\varepsilon>0$ such
  that the ball
  $B_{\varepsilon}(x)\subset
  \mathrm{int}\big(\Lambda_r^v(0)\big)$. Since $\pi(x)$ is an
  essential point for $f$ and we are assuming $f$ is strictly toral,
  by Proposition~\ref{pro:stricly-toral-dynamics} the open set
  $U_{\varepsilon}(\pi(x),f)$ given by \eqref{eq:Ur-set-def} is fully
  essential.

  So, there are simple closed curves
  $\gamma_1,\gamma_2\colon [0,1]\to\T^2$ whose images are contained in
  $U_{\varepsilon}(\pi(x),f)$ and such that they generate the
  fundamental group of $\T^2$. Since $\gamma_1$ and $ \gamma_2$ are
  compact, there exists $m\in\N$ such that
  \begin{displaymath}
    \gamma_1[0,1]\cup\gamma_2[0,1]\subset \bigcup_{j=-m}^m
    f^j\Big(\pi\big(B_\varepsilon(x)\big)\Big). 
  \end{displaymath}

  Now, recalling that
  $\tilde f^n(\Lambda_r^v(0)) \subset \Lambda_{r+n\langle \tilde \rho,
    v \rangle}^v(0)$ for every $n\in\Z$ and the covering
  $\pi\colon\R^2\to\T^2$ is an open map, we conclude that
  \begin{displaymath}
    \gamma_1[0,1]\cup\gamma_2[0,1]\subset 
    \mathrm{int}\big(\pi(\Lambda_{r-m\abs{\scprod{\tilde\rho}{v}}})(0)
    \big).
  \end{displaymath}

  So, we can construct a bi-sequence of simple curves
  $\big(\tilde\gamma^n\colon [0,1]\to\R^2\big)_{n\in\Z}$ satisfying
  the following properties:
  \begin{enumerate}
  \item
    $\tilde\gamma^n(t)\in\Lambda_{r-m\abs{\scprod{\tilde\rho}{v}}}(0)$,
    for every $n\in\Z$ and all $t\in [0,1]$;
  \item $\tilde\gamma^n$ is a lift either of $\gamma_1$ or $\gamma_2$,
    for each $n\in\Z$;
  \item $\tilde\gamma^n(1)=\tilde\gamma^{n+1}(0)$, for every $n\in\Z$;
  \item there exists a constant $M>0$ such
    \begin{displaymath}
      d(\tilde\gamma^n(t),\ell^v_{r-m\abs{\scprod{\tilde\rho}{v}}})\leq
      M, \quad\forall n\in\Z,\ \forall t\in [0,1].
    \end{displaymath}
  \end{enumerate}

  Then, the set $\bigcup_{n}\tilde\gamma^n[0,1]$ clearly separates the
  plane, and its complement has exactly two unbounded connected
  components. In particular, this implies
  $\Lambda_{r-m\abs{\scprod{\tilde\rho}{v}}}^v$ contains a
  semi-plane, and by
  Corollary~\ref{cor:unif-bound-dev-iff-Hm-in-Lambda}, this implies
  $f$ exhibits uniformly bounded $v$-deviations.
\end{proof}

\section{Torus pseudo-foliations and rotational deviations}
\label{sec:tor-foliations}

In this section we introduce the concepts of \emph{pseudo-foliation}
which is a generalization of singularity-free one-dimensioanl plane
foliation.

A \emph{plane pseudo-foliation} is a partition $\F$ of $\R^2$ such
that every atom (also called \emph{pseudo-leaf}) is closed, connected,
has empty interior and separates the plane in exactly two connected
components.

Given a map $h\colon\R^2\carr$, we say that the plane pseudo-foliation
$\F$ is $h$-invariant when
\begin{displaymath}
  h(\F_z)=\F_{h(z)},\quad\forall z\in\R^2.
\end{displaymath}
where $\F_z$ denotes the atom of $\F$ containing $z$; and $\F$ is said
to be $\Z^2$-invariant when it is $T_{\boldsymbol{p}}$-invariant for
any $\boldsymbol{p}\in\Z^2$. 

A \emph{torus pseudo-foliation} is a partition $\F$ of $\T^2$ such
that there exists a $\Z^2$-invariant plane pseudo-foliation $\tilde\F$
satisfying
\begin{displaymath}
  \pi(\tilde\F_z)=\F_{\pi(z)},\quad\forall z\in\R^2.
\end{displaymath}
In such a case, the plane pseudo-foliation $\tilde\F$ is unique and
will be called the \emph{lift} of $\F$.

Notice that a homeomorphism $f\colon\T^2\carr$ leaves invariant a
torus pseudo-foliation $\tilde\F$ if and only if its lift $\tilde\F$
is $\tilde f$-invariant, for any lift $\tilde f\colon\R^2\carr$ of $f$.

Some geometric and topological properties of ``classical'' plane and
torus foliations can be extended to pseudo-foliations:

\begin{proposition}
  \label{prop:plane-pseudo-fol-top-geo-prop}
  If $\F$ is a plane pseudo-foliation, both connected components of
  $\R^2\setminus\F_z$ are unbounded, for every $z\in\R^2$.
\end{proposition}

\begin{proof}
  This easily follows from Zorn's Lemma. In fact, let us suppose there
  exists $z_0\in\R^2$ such that a connected component of
  $\R^2\setminus\F_{z_0}$, called $B_0$, is bounded. So, for every
  $w\in B_0$, $\F_w\subset B_0$ and thus, it is bounded itself. This
  implies that, for all $w\in B_0$, there exists a bounded connected
  component, called $B_w$, of $\R^2\setminus\F_w$, and it clearly
  holds $B_w\subset B_0$, for any $w\in B_0$. 
  
  Now, if we consider the set $\{B_w : w\in B_0 \}$ endowed with the
  partial order given by set inclusion, one can easily check that any
  totally decreasing chain admits a lower bound, and consequently,
  there exists a minimal element. In fact, given a strictly decreasing
  sequence $B_{w_1}\supset B_{w_2}\supset\ldots$ we have that
  $\partial B_{w_n}\cap\partial B_{w_{n+1}}=\emptyset $ because
  different pseudo-leaves are disjoint, and therefore, it holds
  $\bigcap_{n\geq 1} B_{w_n}=\bigcap_{n\geq
    1}\overline{B_{w_n}}\neq\emptyset$. Now, taking any point
  $w\in \bigcap_{n\geq 1} B_{w_n}$, $B_{w}$ happens to be a lower
  bound for the chain. So, by Zorn's Lemma there exists a minimal
  element $B_z$ but of course, $B_w\subsetneq B_z$, for any
  $w\in B_z$, getting a contradiction.
\end{proof}

In order to show that torus pseudo-foliations exhibit some properties
similar to classical foliations, we will use a geometric result due to
Koropecki and Tal~\cite {KoroTalBoundUnbound}. To do that, first we
need to introduce some terminology.

Given a closed discrete set $\Sigma\subset\R^2$, a subset
$U \subset \R^2$ is said to be \emph{$\Sigma$-free} when
$T_{\boldsymbol{p}}(U)\cap U = \emptyset$, for all
$\boldsymbol{p} \in \Sigma$.

A \emph{chain} is a decreasing sequence $\mathscr{C}=(U_n)_{n\in\N}$
of arcwise connected subsets of $\R^2$, \ie $U_{n+1} \subset U_n$, for
all $n\in\N$. We say that $\mathscr C$ is \emph{eventually
  $\Sigma$-free} if for every $\boldsymbol{p}\in\Sigma$, there is
$n\in\N$ such that $T_{\boldsymbol{p}}(U_n)\cap U_n= \emptyset$.

Let us write $\Z_*^2:=\Z^2\setminus\{(0,0)\}$ and recall the following
result of \cite[Theorem 3.2]{KoroTalBoundUnbound}:
\begin{theorem}
  \label{thm:free-chain}
  If $\mathscr C=(U_n)_{n \in \N}$ is an eventually $\Z^2_*$-free
  chain of arcwise connected sets, then one of the following holds:
  \begin{enumerate}[(i)]
  \item\label{eq:chain-almost-free} there exist $n \in \N$ and
    $\boldsymbol{q}\in \Z^2$ such that $U_n$ is
    $\Z^2\setminus(\R\boldsymbol{q})$-free;
  \item\label{eq:chain-one-point-infty} there is a unique
    $v \in \Ss^1$ such that
    $\bigcap_{n \in \N} \partial_\infty U_n=\{v\}$, where
    $\partial_\infty$ denotes the boundary at infinity defined in
    \S\ref{sec:boundary-at-infinity};
  \item\label{eq:chain-separate-annulus} there are $v\in\Ss^1$ and $r>0$ such that
    $\bigcap_{n \in \N} \overline{U_n}$ is contained in the strip
    $\A^v_r$ and it separates the boundary components of $\A_r^v$.
  \end{enumerate}
\end{theorem}

Notice that, if property \eqref{eq:chain-separate-annulus} holds,
then we have $\bigcap_{n \in \N} \partial_\infty U_n=\{-v,v\}$.

Now we will use Theorem~\ref{thm:free-chain} to prove our main result
about torus pseudo-foliations:

\begin{theorem}
  \label{thm:pseudo-fol-top-geo-prop}
  If $\F$ is a $\Z^2$-invariant pseudo-foliation, then there exist
  $v\in\Ss^1$ and $r>0$ such that
  \begin{equation}
    \label{eq:pseudo-leaf-in-strip}
    \F_z\subset T_z(\A_r^v),\quad\forall z\in\R^2.
  \end{equation}
  In such a case, $v$ is unique up to multiplication by $(-1)$ and we
  say that $v^\perp$ is an \emph{asymptotic direction} of $\F$.
	
  Moreover, there is $w \in \R^2$ such that $\F_w$ separates the
  boundary components of the strip $T_w(\A_r^v)$. 
\end{theorem}

\begin{proof}
  Let $z$ denote an arbitrary point of $\R^2$ and, for each $n\in\N$,
  consider the set
  \begin{displaymath}
    U_n(z):=\left\{x\in\R^2 : d(x,\F_z)<\frac{1}{n}\right\}.
  \end{displaymath}
  Notice all the sets $U_n(z)$ are arcwise connected,
  $U_{n+1}(z) \subset U_n(z)$ and it holds
  $\F_z=\bigcap_{n\in\N} U_n(z)$.

  First, let us suppose that there exists $z\in\R^2$ such that the
  chain $\big(U_n(z)\big)_{n \in \N}$ is not eventually
  $\Z^2_*$-free. So, there is $\boldsymbol{p}\in\Z^2_*$, such that
  $T_{\boldsymbol{p}}\big(U_n(z)\big)\cap U_n(z)\neq\emptyset$, for
  every $n\in\N$. By
  Proposition~\ref{prop:plane-pseudo-fol-top-geo-prop}, each
  pseudo-leaf of $\F$ is unbounded in $\R^2$. So, this implies that
  for $n$ sufficiently big,
  $T_{\boldsymbol{p}^\perp}\big(U_n(z)\big)\cap U_n(z) = \emptyset$. Let us
  fix such an $n$, and consider the set
  \begin{displaymath}
    \Theta:=\bigcup_{j\in\Z} T_{\boldsymbol{p}}^j\big(U_n(z)\big).
  \end{displaymath}
  Then $\Theta$ is open, connected, $T_{\boldsymbol{p}}$-invariant and
  $\Theta\cap T_{\boldsymbol{p}^\perp}(\Theta)=\emptyset$. In
  particular, it is contained in the lift of an annular subset of
  $\T^2$, and this implies there exists $r>0$ such that
  $\F_z\subset\A_r^v$. The uniform result for every pseudo-leaf easily
  follows from the fact that
  $\Theta\cap T_{\boldsymbol{p}^\perp}(\Theta)=\emptyset$ and the
  definition of plane pseudo-foliations.

  So, now we can assume that $(U_n(z))_{n \in \N}$ is eventually
  $\Z^2_*$-free for every $z\in\R^2$, and hence,
  Theorem~\ref{thm:free-chain} can be applied.

  First suppose property \eqref{eq:chain-almost-free} of
  Theorem~\ref{thm:free-chain} holds for some $z\in\R^2$, \ie there
  are $n_0\in\N$ and $\boldsymbol{q}\in\Z^2$ such that $U_{n_0}(z)$ is
  $\Z^2\setminus(\R\boldsymbol{q})$-free. Since
  $U_n(z) \subset U_{n_0}(z)$ for $n>n_0$, the same holds for every
  $n>n_0$. Without loss of generality we can assume $\boldsymbol{q}$
  generates the cyclic group $\Z^2\cap \R\boldsymbol{q}$. Then, since
  the chain $(U_n(z))_{n \in \N}$ is eventually $\Z^2_*$-free, there is $n' \in \N$ such that
  $T_{\boldsymbol{q}}\big(U_{n'}(z)\big)\cap U_{n'}(z) = \emptyset$.

  By Theorem~\ref{thm:Franks-free-disk}, we have
  \begin{displaymath}
    T_{\boldsymbol{q}}^j\big(U_{n'}(z))\cap U_{n'}(z) = \emptyset,
    \quad\forall j\in\Z\setminus\{0\}. 
  \end{displaymath}
  So, given any integer $n\geq\max\{n_0,n'\}$, we conclude the chain
  $\big(U_n(z)\big)_{n\geq 1}$ is $\Z^2_*$-free.

  On the other hand, by Proposition
  \ref{prop:plane-pseudo-fol-top-geo-prop} we know the pseudo-leaf
  $\F_z$ is unbounded. Then, there is a sequence
  $(z_j)_{j\in \N} \subset \F_z$ such that the balls
  $B_{1/n}(z_1),B_{1/n}(z_2),\ldots$ are pairwise disjoints.  Then,
  there is a sequence $(\boldsymbol{p}_j)_{j\in\N}\subset\Z^2$ such
  that $T_{\boldsymbol{p}_j}(z_j) \in [0,1]^2$, for every $j\in
  \N$. Since $\big(U_n(z)\big)_{n\geq 1}$ is $\Z^2_*$-free and
  $B_{1/n}(z_j)\subset U_n(z)$ for every $j\in \N$, we get that the
  balls $\big(T_{\boldsymbol{p}_j}(B_{1/n}(z_j))\big)_{j\geq 1}$ are
  pairwise disjoint, and this is clearly impossible. So, the chain
  $\big(U_n(z)\big)_{n \in \N}$ does not verify condition
  \eqref{eq:chain-almost-free} of Theorem~\ref{thm:free-chain}, for
  every $z\in\R^2$. 
  
  So, let us suppose the chain $\big(U_n(z)\big)_{n\geq 1}$ satisfies
  condition \eqref{eq:chain-one-point-infty} of
  Theorem~\ref{thm:free-chain}, for every $z\in\R^2$. Then, for each
  $z\in\R^2$, let $v_z \in\Ss^1$ denote the only point such that
  $\partial_\infty \F_z=\{v_z\}$. Let $\Omega_z$ be the connected
  component of $\R^2\setminus\F_z$ such that
  $\partial_\infty(\Omega_z)=\{v_z\}$. Notice that $\Omega_z$ is an
  open disk, and since the boundary of $\Omega_z$ is contained in the
  pseudo-leaf $\F_z$, we conclude that $\overline{\Omega_z}$ is simply
  connected, for all $z\in\R^2$.

  Then we define the set
  \begin{displaymath}
    \mathcal{M}:=\left\{w\in\R^2 : \forall z\in\R^2,\
    w\not\in\Omega_z\right\}. 
  \end{displaymath}
  Since every $\Omega_z$ is open, $\mathcal{M}$ is clearly closed and
  might be empty.

  Let us define set
  \begin{displaymath}
    \Xi:=\R^2\setminus\left(\bigcup_{w\in\mathcal{M}}
      \overline{\Omega_w}\right) 
  \end{displaymath}
  Observe that if $w$ and $z$ are two elements of $\mathcal{M}$ such
  that $\F_w\neq\F_z$, then it holds
  $\overline{\Omega_z}\cap\overline{\Omega_w}=\emptyset$. So, since
  $\R^2$ cannot be partitioned into countably many (and more than two)
  closed sets, this implies that $\Xi$ is non-empty; and noticing
  $\overline{\Omega_z}$ is simply-connected for every
  $z\in\mathcal{M}$, we can conclude that $\Xi$ is connected. On the
  other hand, since $\mathcal{M}$ is $\Z^2$-invariant, so is $\Xi$. 

  Hence, $\Xi$ is a non-empty open connected $\Z^2$-invariant set and we
  shall consider the following relation on it: given $w_1,w_2\in\Xi$,
  we define
  \begin{displaymath}
    w_1\sim w_2 \iff \exists z\in\R^2,\ w_1,w_2\in\Omega_z.
  \end{displaymath}
  Let us show $\sim$ is an equivalent relation. First notice it is
  clearly reflexive and symmetric. To prove that is transitive too, it
  is enough to observe that given two arbitrary points
  $z_1,z_2\in\R^2$ such that
  $\overline{\Omega_{z_1}}\cap\overline{\Omega_{z_2}}\neq\emptyset$,
  then either $\Omega_{z_1}\subset\Omega_{z_2}$, or
  $\Omega_{z_2}\subset\Omega_{z_1}$. On the other hand, since
  $\Omega_z$ is open for every $z\in\R^2$, we conclude that $\sim$ is
  an open equivalent relation (\ie every equivalent class is
  open). But we had already shown that $\Xi$ is connected, so this
  implies that there is just one equivalent class. Thus, taking any
  point $w\in\Xi$ and any $\boldsymbol{p}\in\Z^2_*$, we get that
  $w\sim T_{\boldsymbol{p}}(w) \sim T_{\boldsymbol{p}^\perp}(w)$. So,
  there should exist a point $z\in\R^2$ such that $\Omega_z$ contains
  the points $w$, $T_{\boldsymbol{p}}(w)$ and
  $T_{\boldsymbol{p}^\perp}(w)$. This clearly contradicts the fact that
  $\partial_\infty\Omega_z$ is a singleton.

  Therefore, it is not possible that condition
  \eqref{eq:chain-one-point-infty} holds for every
  $\big(U_n(z)\big)_{n\geq 1}$. So, there exists $z_0\in\R^2$,
  $v\in\Ss^1$ and $r>0$ such that $\F_{z_0}\subset\A_{r/2}^v$ and it
  separates the connected components of the boundary of the strip
  $\A_{r/2}^v$. Since, the pseudo-foliation $\F$ is $\Z^2$-invariant,
  this clearly implies that condition \eqref{eq:pseudo-leaf-in-strip}
  holds, as desired.
\end{proof}

Let us relate the existence
of pseudo-foliations with the boundedness of rotational deviations:

\begin{theorem}
  \label{thm:inv-pseudo-fol-implies-bounded-dev}
  If $f\in\Homeo0(\T^2)$ leaves invariant a torus pseudo-foliation
  $\F$ and $v^\perp\in\Ss^1$ is an asymptotic direction of (the lift of)
  $\F$, $f$ exhibits uniformly bounded $v$-deviations.

  In particular the rotation set $\rho(f)$ has empty interior.
\end{theorem}

\begin{proof}
  This theorem follows from the combination of
  Theorem~\ref{thm:pseudo-fol-top-geo-prop} and the argument used to
  prove Proposition 4.2 of \cite{KocKorFoliations}. 
  
  Let $\tilde f\in\Thomeo(\T^2)$ be a lift of $f$, and $\tilde\F$ be
  the lift of an $f$-invariant torus pseudo-foliation $\F$. So,
  $\tilde\F$ is $\tilde f$-invariant. Let $v\in\Ss^1$ such that
  $v^\perp$ determines the asymptotic direction of $\tilde \F$. Without
  loss of generality we can assume that $v\in\Ss^1$ is not vertical,
  \ie $\pr{1}(v)\neq 0$.

  Let $r$ be a positive real number given by
  Theorem~\ref{thm:pseudo-fol-top-geo-prop} and let $\tilde \F_w$ be a
  pseudo-leaf of $\tilde \F$ such that $\tilde\F_w\subset T_w(\A_r^v)$
  separates the boundary components of $T_w(\A_r^v)$.

  Since $v^\perp\neq (1,0)$, we know that
  $T_{(1,0)}(\tilde\F_w)\cap\tilde\F_w=\emptyset$. So, we can consider
  the strip
  \begin{displaymath}
    S:=\cc\Big(\R^2\setminus\tilde\F_w,
    T_{(1,0)}\big(\tilde\F_w\big)\Big) \cap \cc\Big(\R^2\setminus
    T_{(1,0)}\big(\tilde\F_w\big), \tilde\F_w\Big).
  \end{displaymath}

  Note that $\overline{S}\subset T_w(\A_{r+1}^v)$ and
  \begin{equation}
    \label{eq:S-int-trans-cover-R2}
    \bigcup_{n \in \Z} T_{(n,0)}(\overline{S})=\R^2.
  \end{equation}

  This implies that for each $n \in \Z$, there exists $m_n\in \Z$ such
  that
  \begin{displaymath}
    \tilde f^n\big(\tilde\F_w\big) \subset T_{(m_n,0)}(\overline{S}),
  \end{displaymath}
  and consequently,
  \begin{displaymath}
    \tilde f^n\Big(T_{(1,0)}\big(\tilde\F_w\big)\Big) =
    T_{(1,0)}\Big(\tilde f^n\big(\tilde\F_w\big)\Big) \subset
    T_{(m_n+1,0)}(\overline{S}). 
  \end{displaymath}
  Hence, we conclude that
  \begin{equation}
    \label{eq:orbit-of-S-v-deviations}
    \tilde f^n(\overline{S})\subset T_{(m_n,0)}\big(\overline{S}\cup
    T_{(1,0)}(\overline{S})\big)\subset T_{(m_n,0)}(\A_{2r+2}^v),
    \quad\forall n\in\Z. 
  \end{equation}

    This fact implies that $\tilde f$ satisfies condition
  \eqref{eq:rot-set-in-line}, \ie, there exists $\alpha\in\R$ such
  that
  \begin{displaymath}
    \rho(\tilde f)\subset\ell_\alpha^v.
  \end{displaymath}

  So we can define the $\tilde{\rho}$-centralized skew-product and the
  $\pm v$-stable sets at infinity as in
  \S~\ref{sec:induced-skew-prod}. By
  Theorem~\ref{thm:translates-of-Lambdas-dense}, there exists $r'>0$
  such that $\Lambda_{-r'}^{-v}(0)\cap S\neq\emptyset$. Let $z_0$ be
  any point in $\Lambda_{-r'}^{-v}(0)\cap S$ and $z$ be an arbitrary
  point of $\R^2$. By \eqref{eq:S-int-trans-cover-R2}, there exists
  $m\in\Z$ such that $T_{(m,0)}(z)\in S$.

  Since $z_0\in\Lambda_{-r'}^{-v}(0)$, it holds
  \begin{equation}
    \label{eq:z0-in-Lambda-r-v}
    \scprod{\tilde f^n(z_0)-z_0}{v} - n\alpha\leq r',\quad\forall
    n\in\Z. 
  \end{equation}
  So putting together \eqref{eq:orbit-of-S-v-deviations} and
  \eqref{eq:z0-in-Lambda-r-v}, we conclude that
  \begin{displaymath}
    \begin{split}
      \scprod{\tilde f^n(z)-z}{v} - n\alpha &= \scprod{\tilde f^n\circ
        T_{(m,0)}(z) - T_{(m,0)}(z)}{v} - n\alpha \\
      &\leq \scprod{\tilde f^n(z_0)-z_0}{v} - n\alpha + 2r + 2 \leq
      2r+ 2 +r',
    \end{split}   
  \end{displaymath}
  and hence, every point exhibits uniformly bounded $v$-deviations.
\end{proof}

The following result is a partial reciprocal of
Theorem~\ref{thm:inv-pseudo-fol-implies-bounded-dev}:

\begin{theorem}
  \label{thm:pseudo-fol-vs-bounded-rot-dev}
  Let $f\in\Homeo0(\T^2)$ be an area-preserving non-annular periodic
  point free homeomorphism, $\tilde f\colon\R^2\carr$ a lift of $f$,
  $\alpha\in\R$, $v\in\Ss^1$ and $M>0$ such that
  \begin{displaymath}
    \abs{\scprod{\tilde f^n(z)-z}{v} - n\alpha}\leq M, \quad\forall
    z\in\R^2. 
  \end{displaymath}

  Then, there exists an $f$-invariant torus pseudo-foliation with
  asymptotic direction equal to $v^\perp$.
\end{theorem}

As it was already shown in \cite[\S4]{JaegerTalIrratRotFact}, the
hypothesis is non-annularity is essential to guaranty the existence of
the invariant pseudo-foliation, but the area-preserving assumption
might be relaxed.

\begin{proof}[Proof of
  Theorem~\ref{thm:pseudo-fol-vs-bounded-rot-dev}]

  First observe that if $f$ is an area-preserving irrational
  pseudo-rotation with uniformly bounded rotational deviations in
  every direction, then J\"ager showed \cite[Theorem
  C]{JaegerLinearConsTorus} that $f$ is a topological extension of
  totally irrational torus rotation. In such a case, the pre-image by
  the semi-conjugacy on any linear torus foliation will yield an
  $f$-invariant pseudo-foliation.

  So, we can assume $f$ is not a pseudo-rotation with uniformly
  bounded rotational deviations in every direction.

  If $v$ has rational slope, taking into account $f$ is non-annular we
  can conclude $\alpha$ is an irrational number and this case is
  essentially considered in Theorem 3.1 of
  \cite{JaegerTalIrratRotFact}. In fact, under these hypotheses it can
  be easily proved that the the family of circloids constructed there,
  which are nothing but the fibers of the factor map over the
  irrational circle rotation, is an $f$-invariant pseudo-foliation.

  So, from now on let us assume $v$ has irrational slope. Let
  $\tilde f\colon\R^2\carr$ be a lift of $f$ and choose an arbitrary
  point $\tilde\rho\in\rho(\tilde f)$. Then we consider the induced
  $\tilde\rho$-centralized skew-product
  $F\colon\T^2\times\R^2\carr$. For each $r\in\R$ and $t\in\T^2$,
  consider the $(r,v)$-fibered stable set at infinity $\Lambda_r^v(t)$
  given by \eqref{eq:Lambda-hat-r-v-definition}.

  For simplicity, let us fix $t=0$. By Corollary
  \ref{cor:unif-bound-dev-iff-Hm-in-Lambda}, there is $M>0$ such that
  $\bb{H}_{r+M}^v \subset \Lambda_r^v(0)$, for all $r \in \R$. Then we
  can define
  
  \begin{displaymath}
    \begin{split}
      U_r&:=\cc\Big(\mathrm{int}\big(\Lambda_r^v(0)\big),
      \bb{H}_{r+M}^v \Big), \\
      C_r&:= \partial U_r,\quad\forall r\in\R,
    \end{split}
  \end{displaymath}
  where $\partial(\cdot)$ denotes the boundary operator in
  $\R^2$. Observe that
  $\overline{U_r}=U_r\cup C_r\subset\Lambda_r^v(0)$, for any
  $r\in\R$. Hence,
  $T_{\boldsymbol{p}}(U_r) = U_{r+\scprod{\boldsymbol{p}}{v}}$ and
  consequently,
  $T_{\boldsymbol{p}}(C_r) = C_{r+\scprod{\boldsymbol{p}}{v}}$, too,
  for every $\boldsymbol{p}\in\Z^2$ and any $r$.
  
  Then we claim the sets $C_r$ are pairwise disjoint. To prove this,
  reasoning by contradiction, let us suppose this is not the case and
  so there exist $s_0<s_1$ such that
  $C_{s_0}\cap C_{s_1}\neq\emptyset$. By monotonicity of the family
  $\{\Lambda_r^v(0) : r\in\R\}$, we get that
  \begin{equation}
    \label{eq:Cs0-s1-interm-intersec}
    \emptyset\neq C_{s_0}\cap C_{s_1}\subset \big(C_{s_0}\cap C_r\big)\cap
    \big(C_{s_1}\cap C_r\big),\quad\forall r\in (s_0,s_1).
  \end{equation}

  Let $r_0$ and $r_1$ be any pair of real numbers such that
  $s_0<r_0<r_1<s_1$ and consider the set
  \begin{displaymath}
    L:=\left\{\boldsymbol{p}\in\Z^2 : \frac{s_0-r_0}{2}<
      \scprod{\boldsymbol{p}}{v}<\frac{s_1-r_1}{2}\right\}.
  \end{displaymath}

  Notice that $L$ has \emph{bounded gaps in both coordinates}, \ie
  there exists $N\in\N$ such that for every $m\in\Z$ it holds
  \begin{equation}
    \label{eq:L-bounded-gaps}
      \{\pr{i}(\boldsymbol{p}) : \boldsymbol{p}\in L\}\cap\{n\in\Z :
      \abs{m-n}\leq N\}\neq\emptyset, \quad\text{for } i\in\{1,2\}. 
  \end{equation}
 
  On the other hand, consider the open set
  $\Omega:=U_{r_0}\setminus \overline{U_{r_1}}$, and observe that
  \begin{equation}
    \label{eq:Tp-Omega-in-s0-s1}
    T_{\boldsymbol{p}}(\Omega)\subset
    U_{s_0}\setminus\overline{U_{s_1}},\quad\forall \boldsymbol{p}\in
    L. 
  \end{equation}

  Putting together \eqref{eq:Cs0-s1-interm-intersec},
  \eqref{eq:L-bounded-gaps} and \eqref{eq:Tp-Omega-in-s0-s1} we can
  see that the diameter of the connected components of $\Omega$ must
  be uniformly bounded, \ie there exists a real number $D>0$ such that
  \begin{equation}
    \label{eq:Omega-cc-diam-unif-bound}
    \mathrm{diam}(\Omega')\leq D,\quad\forall
    \Omega'\in\pi_0(\Omega). 
  \end{equation}

  On the other hand, we know that
  \begin{equation}
    \label{eq:Cr-equivariance}
    \tilde f^n(C_r)=C_{r+n\alpha},\quad\text{and}\quad
    T_{\boldsymbol{p}}(C_r)=C_{r+\scprod{\boldsymbol{p}}{v}},
    \quad\forall n\in\Z,\ \forall r\in\R.
  \end{equation}
  So, putting together \eqref{eq:Omega-cc-diam-unif-bound} and
  \eqref{eq:Cr-equivariance} we can conclude that the set
  \begin{displaymath}
    \left\{n\in\Z : \exists \boldsymbol{p}\in\Z^2,\
      T_{\boldsymbol{p}}\big(\tilde f^n(\Omega)\big)\in
      U_{s_0}\setminus\overline{U_{s_1}}\right\} 
  \end{displaymath}
  has bounded gaps, and consequently, there exists $D'>0$ such that
  \begin{equation}
    \label{eq:Omega-cc-orbit-diam-unif-bound}
   \mathrm{diam}\left(\tilde f^n(\Omega')\right)\leq D',\quad\forall
   n\in\Z. 
  \end{equation}

  Then, putting together the fact that $f$ is strictly toral and
  property \eqref{eq:Omega-cc-orbit-diam-unif-bound}, we immediately
  conclude that $f$ exhibits uniformly bounded rotational deviations
  in every direction, contradicting our original assumption.

  So, we have shown that the open set $\Omega$ separates the closed
  sets $C_{s_0}$ and $C_{s_1}$, and then they do not intersect.

  Then we define the function $H\colon\R^2\to\R$ by
  \begin{displaymath}
    H(z):=\sup\left\{r\in\R : z\in U_r\right\}, \quad\forall
    z\in\R^2. 
  \end{displaymath}
  Since we have shown that the family $\{C_r:r\in\R\}$ is pairwise
  disjoint, it easily follows that the function $H$ is continuous. By
  \eqref{eq:Cr-equivariance}, it follows that 
  \begin{displaymath}
    H\big(\tilde f^n(z)\big) = H(z) + n\alpha, \quad\forall n\in\Z,\
    \forall z\in\R^2,
  \end{displaymath}
  so the partition given by the $H$-level sets is $f$-invariant. On
  the other hand, by the topological properties of the set $U_r$ and
  $C_r$, it clearly follows that each $H$-level set is connected and
  disconnects the plane in two connected components.
  
  So, order to show that the level sets of $H$ determines a
  pseudo-foliation, it just remains to show that, for each $r\in\R$,
  the set $H^{-1}(r)$ has empty interior. To do that, let us suppose
  this is not the case. Thus there exists $r$ such that $H^{-1}(r)$
  has non-empty interior in $\R^2$. Let $W$ be a connected component
  of the interior of $H^{-1}(r)$. Since $H^{-1}(r)$ separates the
  plane in exactly two connected components, we conclude that $W$ is
  an open topological disc. And since the covering map
  $\pi\colon\R^2\to\T^2$ is an open map, $\pi(W)$ will be an open
  itself.

  Then, taking into account $f$ is non-wandering, there exists
  $n_0\geq 1$ such that
  \begin{equation}
    \label{eq:non-empty-int-pseudoleaf-return}
    f^{n_0}(\pi(W))\cap \pi(W)\neq\emptyset.
  \end{equation}
  That means there exists $\boldsymbol{q}\in\Z^2$ such that
  \begin{equation}
    \label{eq:n0-returning-to-W}
    \tilde f^{n_0}\big(W)\big)\cap
    T_{\boldsymbol{q}}(W)\neq\emptyset.  
  \end{equation}
  Since the partition in level sets of $H$ is $\tilde f$-invariant,
  this implies
  $\tilde f^{n_0}\big(H^{-1}(r)\big)=
  T_{\boldsymbol{q}}\big(H^{-1}(r)\big)$, and taking into account $W$
  is a connected component of the interior of $H^{-1}(r)$, we
  conclude that $\tilde f^{n_0}(W)=T_{\boldsymbol{q}}(W)$. So,
  \eqref{eq:n0-returning-to-W} can be improved: in fact, it holds
  $f^{n_0}(W)=W$. On the other hand, since $v$ has irrational slope,
  we know that
  \begin{displaymath}
    T_{\boldsymbol{p}}\big(H^{-1}(r)\big)\cap
    H^{-1}(r)=\emptyset,\quad\forall
    \boldsymbol{p}\in\Z^2\setminus\{0\}. 
  \end{displaymath}
  This implies $\pi(W)$ is an open disc in $\T^2$, and since $f$ is
  non-wandering, by Theorem~\ref{thm:Brouwer} $f^{n_0}$ should have a
  fixed point on $\pi(W)$, contradicting the fact that $f$ is periodic
  point free.
\end{proof}

\bibliographystyle{amsalpha} \bibliography{base-biblio}

\providecommand{\bysame}{\leavevmode\hbox to3em{\hrulefill}\thinspace}
\providecommand{\MR}{\relax\ifhmode\unskip\space\fi MR }
\providecommand{\MRhref}[2]{%
  \href{http://www.ams.org/mathscinet-getitem?mr=#1}{#2}
}
\providecommand{\href}[2]{#2}
\begin{thebibliography}{GKT14}

\bibitem[AZ15]{Addas-ZanataUnifBoundDiff}
S.~Addas-Zanata, \emph{Uniform bounds for diffeomorphisms of the torus and a
  conjecture of boyland}, Journal of the London Mathematical Society
  \textbf{91} (2015), no.~2, 537--553.

\bibitem[BCJ17]{BeguinCrovisierJaegerADynDecomp}
F.~{B\'eguin}, S.~Crovisier, and T.~{J\"ager}, \emph{A dynamical decomposition
  of the torus into pseudo-circles}, Modern theory of dynamical systems,
  Contemp. Math., vol. 692, Amer. Math. Soc., Providence, RI, 2017, pp.~39--50.

\bibitem[Bir34]{BirkhoffNouvelles}
G.~D. Birkhoff, \emph{Nouvelles recherches sur les syst{\`e}mes dynamiques}, Ex
  aedibus academicis in Civitate Vaticana, 1934.

\bibitem[D\'13]{DavalosTorusHomeoRotInt}
P.~D\'avalos, \emph{On torus homeomorphims whose rotation set is an interval},
  Mathematische Zeitschrift \textbf{275} (2013), 1005--1045.

\bibitem[D{\'a}v13]{DavalosAnnularMapsTorus}
P.~D{\'a}valos, \emph{On annular maps of the torus and sublinear diffusion},
  Preprint arXiv:1311.0046, 2013.

\bibitem[Fat87]{FathiOrbCloBrouwer}
A.~Fathi, \emph{An orbit closing proof of {B}rouwer's lemma on translation
  arcs}, Enseign. Math. (2) \textbf{33} (1987), no.~3-4, 315--322. \MR{925994
  (89d:55004)}

\bibitem[Fra92]{FranksNewProofBrouwer}
J.~Franks, \emph{A new proof of the {Brouwer} plane translation theorem},
  Ergodic Theory and Dynamical Systems \textbf{12} (1992), no.~02, 217--226.

\bibitem[GKT14]{GuelKorTalAnnAreaPresToral}
N.~Guelman, A.~Koropecki, and F.~Tal, \emph{A characterization of annularity
  for area-preserving toral homeomorphisms}, Mathematische Zeitschrift
  \textbf{276} (2014), no.~3-4, 673--689.

\bibitem[Han89]{HandelPerPointFree}
M.~Handel, \emph{Periodic point free homeomorphism of {$\mathbb{T}^2$}}, Proc.
  Amer. Math. Soc. \textbf{107} (1989), no.~2, 511--515. \MR{965243
  (90c:58168)}

\bibitem[Her79]{HermanSurLaConj}
M.~Herman, \emph{Sur la conjugaison diff\'erentiable des diff\'eomorphismes du
  cercle \`a des rotations}, Inst. Hautes \'Etudes Sci. Publ. Math. \textbf{49}
  (1979), 5--233.

\bibitem[J{\"a}g09]{JaegerLinearConsTorus}
T.~J{\"a}ger, \emph{Linearization of conservative toral homeomorphisms},
  Inventiones Mathematicae \textbf{176} (2009), no.~3, 601--616. \MR{2501297}

\bibitem[JT17]{JaegerTalIrratRotFact}
T.~J\"ager and F.~Tal, \emph{Irrational rotation factors for conservative torus
  homeomorphisms}, Ergodic Theory and Dynamical Systems \textbf{37} (2017),
  no.~5, 1537--1546.

\bibitem[JZ98]{JonkerZhangTorHomRot}
L.~Jonker and L.~Zhang, \emph{Torus homeomorphisms whose rotation sets have
  empty interior}, Ergodic Theory Dynam. Systems \textbf{18} (1998), no.~05,
  1173--1185.

\bibitem[KK08]{KocKorFreeCurves}
A.~Kocsard and A.~Koropecki, \emph{Free curves and periodic points for torus
  homeomorphisms}, Ergodic Theory \& Dynamical Systems \textbf{28} (2008),
  1895--1915.

\bibitem[KK09]{KocKorFoliations}
\bysame, \emph{A mixing-like property and inexistence of invariant foliations
  for minimal diffeomorphisms of the 2-torus}, Proceedings of the American
  Mathematical Society \textbf{137} (2009), no.~10, 3379--3386.

\bibitem[Koc16]{KocMinimalHomeosNotPseudo}
A.~Kocsard, \emph{On the dynamics of minimal homeomorphisms of $\mathbb{T}^2$
  which are not pseudo-rotations}, arXiv preprint arXiv:1611.03784, 2016.

\bibitem[KPS16]{KoropeckiPasseggiSambarino}
A.~Koropecki, A.~Passeggi, and M.~Sambarino, \emph{The {Franks-Misiurewicz}
  conjecture for extensions of irrational rotations}, preprint
  arXiv:1611.05498, 2016.

\bibitem[KT14a]{KoropeckiTalAreaPrsIrrotDiff}
A.~Koropecki and F.~Tal, \emph{Area-preserving irrotational diffeomorphisms of
  the torus with sublinear diffusion}, Proceedings of the American Mathematical
  Society \textbf{142} (2014), no.~10, 3483--3490.

\bibitem[KT14b]{KoroTalBoundUnbound}
\bysame, \emph{Bounded and unbounded behavior for area-preserving rational
  pseudo-rotations}, Proceedings of the London Mathematical Society
  \textbf{109} (2014), no.~3, 785--822.

\bibitem[KT14c]{KoroTalStricTor}
A.~Koropecki and Fabio~A. Tal, \emph{Strictly toral dynamics}, Invent. Math.
  \textbf{196} (2014), no.~2, 339--381. \MR{3193751}

\bibitem[LCT15]{LeCalvezTalForcingTheory}
P.~Le~Calvez and F.~Tal, \emph{Forcing theory for transverse trajectories of
  surface homeomorphisms}, ArXiv preprint arXiv:1503.09127, 2015.

\bibitem[MZ89]{MisiurewiczZiemian}
M.~Misiurewicz and K.~Ziemian, \emph{Rotation sets for maps of tori}, Journal
  of the London Mathematical Society \textbf{2} (1989), no.~3, 490--506.

\bibitem[Poi80]{Poincare1880memoire}
H.~Poincar{\'e}, \emph{M{\'e}moire sur les courbes d{\'e}finies par les
  {\'e}quations diff{\'e}rentielles i--vi, oeuvre i}, Gauthier-Villar: Paris
  (1880), 375--422.

\end{thebibliography}

\end{document}